\documentclass[reqno]{amsart}
\usepackage{latexsym,amsfonts,amssymb,amsmath,amscd,graphics}
\usepackage{amsthm}
\usepackage[mathscr]{eucal}
\usepackage{mathrsfs}
\usepackage{mathbbol}
\def\bidomain0{bi-\domain0}
\def\bidomains0{bi-\domains0}
\def\bisemifield0{bi-\semifield0}
\def\bisemifields0{bi-\semifields0}
\def\BB{\{0,1 \}}
\def\predomain0{pre-\domain0}
\def\predomains0{pre-\domains0}

\def\ddispace{\setlength{\itemsep}{2pt}}

\usepackage{ifthen}

\def\ZO{$\{0,1\}$}

\def\mfa{\mathfrak a}

\newcommand{\TMM}[1]{\operatorname{TM #1}}
\newcommand{\LV}[1]{\operatorname{LV #1}}
\newcommand{\D}[1]{\operatorname{D #1}}

\def\Lstr{L^\times}
\def\Wstr{W^\times}
\def\Lz{L}

\def\semiring{semiring}

\def\ldsLRpr{(\R', L',   s', (\nu'_{m,\ell} ))}

\def\Cong{\Omega}

\def\R {R}

\newcommand{\ds}[1]{\ {#1} \ }
\newcommand{\dss}[1]{\quad {#1} \quad }

\def\ldsR{(\R, L ,s, (\nu_{m,\ell} )) }

\def\RLsnu{(\R, L, s, (\nu_{m,\ell}))}
\def\RLsnuT{(\R', L',  s', (\nu'_{m',\ell'}))}

\def\pipeWL{{\underset{L}{\mid}}}

\def\lmodWL{\mathrel  \pipeWL   \joinrel \joinrel =}
\def\lmodWLnu{\mathrel  \pipeWL   \joinrel \joinrel \equiv_\nu}

\def\psil{\psi_\ell}

\newcommand\zL{L^0}
\newcommand{\cR}{\mathcal{R}}

\newcommand\boxtext[1]{\pSkip \qquad \qquad \qquad \framebox{\parbox{\ltw}{#1}}\pSkip}
\def\ltw{0.7\textwidth}

\newcommand{\xl}[2]{\,\,{^{[#2]}}{#1}\,}

\def\mfa{\frak a}

\def\lzero{0_L}

\usepackage{epsfig,latexsym,amsfonts,amssymb,amsmath,amscd,graphics,epic}
\usepackage{amsfonts,amssymb,amsmath,amscd,amsthm}
\usepackage[mathscr]{eucal}
\usepackage{mathrsfs}
\usepackage{oldgerm,units}
\usepackage{wrapfig,epsfig}

\usepackage{ifthen}
\usepackage{mathbbol}

\usepackage{amsthm}
\usepackage[mathscr]{eucal}
\usepackage{mathrsfs}
\usepackage{mathbbol}
\usepackage{oldgerm,units}
\usepackage{wrapfig}

\newtheorem{theorem}{Theorem}[section]

\newtheorem{example}[theorem]{Example}

\newcommand{\Real}{\mathbb R}
\newcommand{\Rati}{\mathbb Q}

\newcommand{\Net}{\mathbb N}

\newcommand{\Fld}{\mathbb K}

\newcommand{\one}{\mathbb{1}}
\newcommand{\zero}{\mathbb{0}}

\newcommand{\got}[1]{\frak{#1}}

\newcommand{\trop}[1]{\mathcal{#1}}

\newcommand{\tA}{\trop{A}}

\newcommand{\tF}{\trop{F}}
\newcommand{\tG}{\trop{G}}

\newcommand{\tM}{\trop{M}}

\newcommand{\tT}{\trop{T}}

\newcommand{\To}{\longrightarrow }

\newcommand{\al}{\alpha}

\newcommand{\lm}{\lambda}

\newcommand{\nVal}{Val}

\newcommand{\id}{\inva{\aad}}

    \ifx\proof\undefined
    \newenvironment{proof}{
    \smallskip
    \noindent\emph{Proof.}}{\hfill\(\Box\)
    \bigskip
    } \fi

\newcommand{\bfem}[1]{\textbf{\emph{#1}}}

\newcommand{\ifdef}[3]{\ifthenelse{\equal{#1}{true}}{#2}{#3}}

\def\({\left(}
\def\){\right)}
 \input amssymb.sty
\input xy
\xyoption{all} \textwidth 160mm \textheight 228mm \topmargin -5mm
\def\Box{\operatorname{Box}}

\def\stack{component set}
\def\stak(f){\got{P}}
\def\stak{\got{S}}

\def\vrp{\varphi}
\def\Phiv{\tilde v}
\def\Phiw{\widetilde {w}}

\def\tlv{\tilde v}

\def\nVal{\operatorname{Val}}

\def\lv{\operatorname{s}}

\def\nucong{\cong_{\nu}}

\def\({\left(}
\def\){\right)}

\def\semiring0{semiring$^\dagger$}
\def\semialg0{semi-algebra$^\dagger$}
\def\semirings0{semirings$^\dagger$}
\def\domain0{domain$^\dagger$}
\def\domains0{domains$^\dagger$}
\def\field0{semifield$^\dagger$}
\def\semifield0{semifield$^\dagger$}
\def\semifields0{semifields$^\dagger$}

\def\module0{module$^\dagger$}
\def\modules0{modules$^\dagger$}

\def\semifield0{semifield$^\dagger$}

\def\nucong{\cong_\nu}
\def\nnucong{\not \cong_\nu}

\newcommand{\etype}[1]{\renewcommand{\labelenumi}{(#1{enumi})}}
\def\eroman{\etype{\roman}}
\def\ealph{\etype{\alph}}

\def\pipeGS{{\underset{\operatorname{\, gs }}{\mid}}}

\def\lmodg{\mathrel   \pipeGS \joinrel\joinrel \joinrel =}

\def\pipeL{{\underset{{L}}{\mid}}}
\def\lmodL{\mathrel  \pipeL \joinrel \joinrel =}

\def\ealph{\etype{\alph}}

\def\pSkip{\vskip 1.5mm \noindent}

\def\diag{\operatorname{diag}}
\def\a{\alpha}
\newtheorem{thm}[theorem]{Theorem}

\newtheorem*{thm*}{Theorem}

\newtheorem{cor}[theorem]{Corollary}

\newtheorem{lem}[theorem]{Lemma}
\newtheorem{rem}[theorem]{Remark}
\newtheorem{prop*}{Proposition}

\newtheorem{prop}[theorem]{Proposition}
\newtheorem{defn}[theorem]{Definition}
\newtheorem{construction}[theorem]{Construction}

\newtheorem*{examp*}{Example}
\newtheorem*{examples*}{Examples}
\newtheorem*{remark*}{Remark}

\newtheorem{Note}[theorem]{Note}
\newtheorem*{defn*}{Definition}

\def\TropfunoneL{\tF_{\operatorname{LTrop}}}

\def\tT{\mathcal T}

\numberwithin{equation}{section}

\def\M0{M_{\zero}}

\def\id{\operatorname {id}}

\def\SR{R}

\def\PS{P}

\def\wone{{\one_W}}
\def\wzero{{\zero_W}}

\def\scrR{\mathscr R}

\def\rzero{\zero_\SR}

\def\rone{{\one_\SR}}
\def\mone{{\one_\tM}}

\def\rzero{\zero_\SR}

\def\Valmon{\operatorname{ValMon}}
\def\SValmon{\operatorname{ValMon}^+}

\def\PrePOMon{\operatorname{PPreOMon}}
\def\PreOMon{\operatorname{PreOMon}}
\def\POMon{\operatorname{POMon}}
\def\SOMon{\operatorname{OMon}^+}
\def\OMon{\operatorname{OMon}}

\def\SeR{\operatorname{Semir}}

\def\SeRo{\SeR^\dagger}

\def\LaySR{\operatorname{LaySemi}^\dag}

\def\zLaySR{\operatorname{LaySemi}}

\newcommand{\nPS}[1]{\PS_{(!#1)}}
\newcommand{\nPSo}[1]{\nPS{\one}}

\def\bfa{ \textbf{a}}

\begin{document}


\title[Categories of layered semirings]
{Categories of layered semirings}

\author[Z. Izhakian]{Zur Izhakian}
\address{Department of Mathematics, Bar-Ilan University, Ramat-Gan 52900,
Israel}
\email{zzur@math.biu.ac.il}

\author[M. Knebusch]{Manfred Knebusch}
\address{Department of Mathematics, University of Regensburg, Regensburg,
Germany} \email{manfred.knebusch@mathematik.uni-regensburg.de}

\author[L. Rowen]{Louis Rowen}
\address{Department of Mathematics, Bar-Ilan University, Ramat-Gan 52900,
Israel} \email{rowen@macs.biu.ac.il}

\thanks{This research of the first and third authors is supported  by the
Israel Science Foundation (grant No.  448/09).}

\thanks{This research of the first author has been supported  by the
Oberwolfach Leibniz Fellows Programme (OWLF), Mathematisches
Forschungsinstitut Oberwolfach, Germany.}

\thanks{The second author was supported in part by the Gelbart Institute at
Bar-Ilan University, the Minerva Foundation at Tel-Aviv
University, the Mathematics Dept. of Bar-Ilan University, and the
Emmy Noether Institute.}


\date{\today}





\begin{abstract} We generalize the constructions of \cite{IzhakianKnebuschRowen2009Refined,IKR4} to
 layered semirings, in order to enrich the structure
and provide finite examples for applications in arithmetic
(including finite examples). The layered category theory of
\cite{IKR4} is extended accordingly, to cover noncancellative
 monoids.
\end{abstract}

\maketitle




\section{Introduction}
\numberwithin{equation}{section}

This paper is a continuation of
\cite{IzhakianKnebuschRowen2009Refined} and \cite{IKR4}. Tropical
mathematics often involves the study of  valuations, whose targets
are ordered Abelian groups that can be viewed as max-plus
algebras.  The layered supertropical domain was introduced in
~\cite{IzhakianKnebuschRowen2009Refined}, and put in a categorical
framework in \cite{IKR4}, in order to provide algebraic tools with
which to study  this structure.

Tropical mathematics often involves
 the study of  valuations, whose targets are
ordered Abelian groups that can be viewed as max-plus algebras.
The basic functor used in \cite{IKR4}
  goes from   the category of cancellative ordered Abelian
monoids to the category of $L$-layered \domains0 with respect to a
\semiring0~$L$. (We use the generic notation $^{\dagger}$ to
indicate that we do not require a zero element.)

On the other hand, many important classical arithmetical results
are proved by passing to finite structures (i.e., modulo a prime
number). The main objective of this paper is to open the way to an
arithmetic tropical theory, by permitting finite tropical
structures. This might seem to be an oxymoron, since all
nontrivial ordered groups are infinite. But valuation theory has
been enriched in \cite{HV} and \cite{Z}
by means of valuations to arbitrary
ordered Abelian monoids, thereby raising the possibility of a
layered \semiring0 construction for any
 ordered Abelian monoid.

 Since any
cancellative ordered monoid
 is necessarily infinite,  we need to include noncancellative monoids in our category if we
 want to deal with finite structures and their corresponding arithmetic.
 But then, as observed already in ~\cite{IzhakianKnebuschRowen2009Refined}, the naive analog of
\cite[Construction~3.2]{IzhakianKnebuschRowen2009Refined} does not
satisfy distributivity, so we must turn to a more sophisticated
version, given below in Construction~\ref{defn5}. This requires a
0-layer, i.e., $0\in L,$ at the cost of a decidedly more
complicated multiplication. So at the outset we consider
`absorption' via the elements $\zero$ and $\infty$.

The dividend is far greater flexibility in our examples,  cast in
a more general categorical setting than given in \cite{IKR4}.
Construction~\ref{defn5}  is verified in Theorem~\ref{maxplus23}.
In the process we obtain finite structures, as indicated in
Example~\ref{trun0001}. Namely, applying ``truncation'' both to
the given valued monoid and the sorting set yields finite examples
and could permit one to apply the corresponding arithmetic tools.
Since the 0 layer and infinite layer both play significant roles
in this theory, we study their properties and interactions in
detail in \S\ref{mor1}.

Several serious technical difficulties arise when we try to put
this more general construction in its categorical context, because
a homomorphism of monoids might send a noncancellative monoid to a
cancellative monoid, thereby requiring us to switch back and forth
from one construction to the other. The corresponding maps
apparently cannot be written as morphisms of \semirings0, so  one
must broaden either the class of morphisms or the class of objects
in the category.

We try both approaches in turn, the first approach occupying the
body of this paper and the second approach discussed in the
appendix.  In  \S\ref{mor1} we introduce  our main examples. In
\S\ref{sec:Morphisms}, which still is not cast in full generality,
we pass from the category of valued monoids to the category of
\semirings0 by means of the ``0-excepted'' homomorphisms of
Definition~\ref{gsmorph}. This permits us in \S\ref{supfun} to
describe the tropicalization functor more generally, for rings
that need not be integral domains. Furthermore, if one turns to
the basic link of tropical geometry with classical algebraic
geometry via valuations, one is led to consider more general
``transmissions'' which pass from valuation to valuation.

The key to  tying this in with tropicalization is Kapranov's
Lemma. Elaborating on \cite[\S8]{IKR4}, we show in
Remark~\ref{Puis} how Kapranov's Lemma can be expressed in terms
of a \textbf{Kapranov map}, thereby
 yielding a ``layering''
functor for polynomial functions. This map is compatible with the
tropicalization map given in \cite{Pay1}.

In \S\ref{layval} we see that the category $\LaySR$ of
layered semirings  ties in to layered supervaluations, and
specializes to the category STROP from \cite{IKR5}, when we take $L = \{0, 1,
\infty \}$ and $R_0 = \{ \zero_R \}.$ The key result in this
regard is Theorem~\ref{thm5.4} and its corollary, which show that
the transmissions of layered supervaluations often become layered
homomorphisms under certain natural assumptions.

In Appendix A (\S\ref{mor2})  even fuller generality is obtained
by considering structures more general than \semirings0, analogous
to the supertropical monoids of \cite{IKR5}. Here the
noncancellative products belong to the 0-layer, for which addition
with the rest of the structure is not defined; as a result, we do
not quite have a  \semiring0.

\section{Background}

For us, a monoid is a multiplicative semigroup with a unit element
$\one_\tM.$ We work with semirings and their (multiplicative)
monoids.

\subsection{Semigroups and semirings}

We review a few definitions from semigroups and semirings. We say
that an element $a$ of a semigroup $\tM := (\tM, \cdot \; )$ is
\textbf{partially absorbing} if $ab  = a$ for some $b \in \tM$; an
element $a$ of $\tM := (\tM, \cdot \; )$ is \textbf{absorbing} if
$ab = ba = a$ for all $b\in \tM$.

\begin{lem}\label{abso} If $\tM $ is an Abelian semigroup with a unique partially absorbing
element $a$, then $a$ is absorbing.
\end{lem}
\begin{proof} Suppose  $ab  = a$. For any $c\in \tM$ we have
$(ca)b = c(ab) = ca = ac.$  Thus, $ac$ is partially absorbing,
implying $ac = a$ by hypothesis.
\end{proof}

Usually, the absorbing element is identified with $\zero$, but  it
could also be identified with the \textbf{infinite element}
$\infty,$ given by
\begin{equation}\label{inf}   \infty\cdot a =
a\cdot \infty = \infty , \qquad \text{for all } 0 \ne a \in \tM.
\end{equation}
(We do not necessarily assume that $\tM$ contains $\zero$ or
$\infty$. The partially absorbing element $\infty$ is absorbing
when $\zero \notin \tM$.)

 A
semigroup $\tM$ is \textbf{pointed} if it has an absorbing element
$\zero_\tM$.

 A~semigroup
$\tM$ is \textbf{cancellative} with respect to a subset $S$ if $as
= bs$ implies $a=b$ whenever $a,b \in \tM$ and $s\in S.$
 A pointed semigroup $\tM$
is \textbf{cancellative}  if $\tM$ is cancellative with respect to
$\tM \setminus  \{ \zero_\tM \}. $

An  element  $\infty$ in  a \semiring0 $R$  is \textbf{infinite}
if it is absorbing with respect to addition, i.e.,
 satisfies
\begin{equation}\label{inf1} r + a = r   \qquad \text{for all } a \in
R.
\end{equation}

\begin{defn}
A   \textbf{\domain0}  is a \semiring0 $R$ that is cancellative
under multiplication.  A \semiring0 $R$ is a \textbf{domain} if $R
\cup \{ \zero \}$ is a \domain0. Likewise, $R$ is a
\textbf{\semifield0} if $R$ is closed under multiplication. $R$ is
a \textbf{semifield} if $R \cup \{ \zero \}$ is a \semifield0.
\end{defn}

Although we have two usages for `infinite,' one additive and one
multiplicative, they are connected by the following observation:

\begin{prop}\label{7.2} If $R_\infty = R \cup \{\infty\}$ where $R$ is  a \semifield0
and $\infty \in R_\infty$ is an infinite element  in the sense of
\eqref{inf1},  then $a:= \infty $ also satisfies \eqref{inf}.
\end{prop}
\begin{proof} $a = a + ab = ( ab^{-1} + a)b  = a b .$
\end{proof}

Congruences over \semifields0 are described in detail in
\cite{HW}. (The \domains0 of eventual tropical interest to us are
polynomial \semirings0 over \semifields0, which are needed to
define tropical varieties, as described in
\cite{IzhakianKnebuschRowen2009Refined,IKR4}.

As in \cite{IKR4} we work with the category $\SeRo$ of \semirings0
and their homomorphisms, as compared to the category $\SeR$ of
semirings and semiring homomorphisms. We refer the reader to
\cite{IKR4} for preliminary facts that we need; an earlier
reference is \cite{Cos}. As noted in \cite{IKR4}, the category
$\SeRo$ is isomorphic to a subcategory of the category $\SeR$,
since any \semiring0 $R$ can be embedded in a semiring $R\cup
\{\zero\}$ by formally adjoining a zero element ~$\zero$.

\subsubsection{Pre-ordered semigroups and \semirings0}

\begin{defn}\label{ordered1} A semigroup $\tM := (\tM, \cdot \; )$ (or a monoid $\tM := (\tM, \cdot \; , \one_\tM)$) is \textbf{pre-ordered} (resp.~\textbf{ partially
pre-ordered,         partially ordered}, \textbf{ordered}) if it
has a  pre-order $\le$ (resp. ~ partially pre-order, partial
order, order) such that
\begin{equation}\label{dist10} b\le c \quad \text{implies}
\quad ab \le ac \quad \text{and}  \quad ba \le ca, \quad \forall a
\in \tM.\end{equation}
\end{defn}

As in \cite{IKR4}, we assume that all preorders are positive.
$\PrePOMon $, $\PreOMon $,  $\POMon $, $\OMon $, and $\SOMon $
denote the respective categories
 of partially
pre-ordered, pre-ordered, partially ordered,  ordered, and
cancellative ordered monoids, whose morphisms are the
order-preserving homomorphisms.

The crucial observation here is that any \semiring0  becomes a
partially pre-ordered semigroup via the rule (also cf.~\cite{HW}):
\begin{equation}\label{trans1}\text{$a \le b$ \dss{iff}  $a=b\quad$ or
$\quad b = a+c\quad$ for some $ c \in R.$}\end{equation}

 We say that a \semiring0   $R$ is  \textbf{pre-ordered} (resp.~\textbf{ partially
pre-ordered, partially ordered}, \textbf{ordered}) if it has a
partial pre-order $\le$ (resp. ~partial  order, order) with
respect to which  both the monoid $(R,\cdot \; ,\rone)$ and the
semigroup $(R,+)$ satisfy Condition \eqref{dist10} of
Definition~\ref{ordered1}.

By \cite[Proposition 3.9]{IKR4}, there is a natural functor $\SeRo
\to \PrePOMon,$ where we define the partial pre-order on a
\semiring0 $R$ as in \eqref{trans1}.

\subsection{Valued monoids}\label{ssec:Mono}
Although we focused on ordered monoids in \cite{IKR4}, tropical
mathematics is concerned with valuations. More generally, we can
take the
  target to be a monoid,
 cf.~\cite[Definition~2.1]{IzhakianKnebuschRowen2009Valuation}.

\begin{defn}\label{def:valuedMonoid} A monoid $\tM : = (\tM,\cdot \;, \mone)$
is \textbf{m-valued} with respect to an ordered  monoid~$\tG: =
(\tG,\cdot \; , \geq, \one_\tG)$  if there is an onto monoid
homomorphism $v : \tM \to \tG$.  (In other words, $v(ab) =
v(a)v(b).$)  We also call $v$ an $m$-\textbf{valuation}.
 We notate this set-up as the \textbf{triple}
$(\tM,\tG,v)$.
\end{defn}

 This fits in better with our algebraic
 notation for \semirings0. Thus, any valuation $v: K \to \tG$ is an
 m-valuation, where we just disregard addition in $K$.
The hypothesis that $v$ is onto can always be attained by
replacing $\tG$ by $v(\tM)$ if necessary.

The category of triples should be quite intricate, since the
morphisms should include all maps which ``transmit'' one
m-valuation to another. We explore this idea further in
\S\ref{layval}, but for the most part take a simpler approach,
following \cite{IKR4}.
 \begin{defn}
 $\Valmon$ is the category of valued monoids whose objects are triples $(\tM,
\tG, v)$ as in~Definition~\ref{def:valuedMonoid}, for which a
morphism \begin{equation}\label{eq:valMonMor} \phi : (\tM, \tG, v)
\To (\tM', \tG', v')\end{equation} is comprised of a pair
$(\phi_\tM, \phi_\tG)$ of a monoid homomorphism $\phi_\tM: \tM \to
\tM'$, as well as an order-preserving monoid homomorphism
$\phi_\tG: \tG \to \tG'$, satisfying the compatibility condition
\begin{equation}\label{comm1} v'(\phi_\tM (a)) = \phi_\tG (v(a)),
\quad \forall a \in \tM.\end{equation}
 \end{defn}

\begin{rem}\label{inher} When the value map $v$ of the  triple $(\tM,
\tG, v)$ is 1:1, then $\tM$ inherits the order from $\tG$, by
stipulating that $a<b$  when $v(a)< v(b).$ In this way, we can
view $\OMon$ as a full subcategory of $\Valmon$. \end{rem}

\subsection{Congruences} Since we work in the framework of universal algebras, we need some
general observations, and then specialize to the cases of interest
to us (semigroups and semirings). One defines a
 \textbf{congruence} $\Cong $ of an algebraic structure $\tA$ to be an equivalence
  relation $\equiv$ which preserves
 all the relevant operations and relations;  we call  $\equiv$ the \textbf{underlying  equivalence} of $\Cong $.
  Equivalently, a congruence~$\Cong$ is a sub-structure of $\tA \times \tA$ that contains the
  diagonal $\diag (R):= \{ (a,a): a \in R \}$, as described in Jacobson~\cite[\S2]{Jac}.

 Since
the most important \semirings0 for us are \domains0, we want to
know, given a congruence $\Cong$ on~$R$, when the factor
\semiring0 $R/\Cong $ has an absorbing element, and when it is a
\domain0. Given a subset $A \subset R$, we write $b \equiv A$ if
$b \equiv a$ for some $a \in A$. We call an ideal $\mfa
\triangleleft R$ \textbf{closed} under $\Cong $ if $b \equiv \mfa$
implies $b \in \mfa$.
\begin{lem}\label{pointed} Suppose $\Cong $ is a congruence on a \semiring0  $R.$

\begin{enumerate}\eroman

\item
 $R/\Cong $ is a \domain0 iff its underlying  equivalence  $\equiv $ is \textbf{cancellative}, in the sense that
$ab \equiv ac$ implies  $b \equiv c$. \pSkip

\item If $R/\Cong $ is a semiring with absorbing element, which we
denote as $\bar \zero$, then the pre-image $I $ of $\bar \zero$ is
a closed ideal of $R$ all of whose elements are equivalent.
Conversely, if $\mfa$ is a closed ideal of $R$ all of whose
elements are equivalent, then the image of $\mfa$ is the absorbing
element of $R/\Cong $. \pSkip

\item When (ii) holds,  $R/\Cong $ is a domain iff $\equiv$    is
cancellative with respect to all elements not in $\mfa$, in the
sense that if $  a    b \equiv a     c$ for $a \notin \mfa,$ then
 $  b \equiv    c$.
\end{enumerate}
\end{lem}
\begin{proof} Write  $\bar a$ for the image of $a$ in  $R/\Cong $.

\begin{enumerate}\eroman

\item $ab \equiv ac$ iff  $\bar a\bar b = \bar a \bar c$, iff
$\bar b = \bar c$, iff $b \equiv c$. \pSkip

\item If $a,b \in I,$ then $\bar a = \bar b = \bar \zero$,
implying $a \equiv b$.  Conversely, if $\mfa$ is a closed ideal of
$R$ all of whose elements  are equivalent, then the image of
$\mfa$ is an ideal of $R/\Cong $ consisting of a single element,
which must thus be the absorbing element. \pSkip

\item The condition translates to saying  that $\bar a   \bar b =
\bar a   \bar c$ for $\bar a  \ne \bar \zero$ implies  $  \bar b =
c$.
\end{enumerate}
\end{proof}

It is useful to weaken the notion of congruence.

 \begin{defn}  A \textbf{half-congruence}
 $\Cong$ is a sub-structure of $\tA \times \tA$ that contains the diagonal and is transitive
 in the sense that if $\Cong$ contains $(a,b)$ and $(b,c)$ then it also contains $(a,c).$\end{defn}

Throughout the body of this paper $R$ denotes a commutative \semiring0.

\begin{example} In the language of monoids, if $\mfa_1, \mfa_2 $ are monoid ideals of a monoid $\tM := (\tM, \cdot \; ),$ then
$$(\mfa_1\times \mfa_2) \cup \{(a,a): a \in \tM\}$$ is a  congruence since $\mfa_ia \subseteq \mfa_i$.
But in the language of \semirings0, if $\mfa_1, \mfa_2 $ are
\semiring0 ideals of a \semiring0 $R,$ then $(\mfa_1\times \mfa_2)
\cup \{(r,r): r \in R\}$ need not even be a half-congruence, since it
may not be closed under addition. (In general, $\mfa_i + r \not
\subseteq \mfa_i.$)
\end{example}

\begin{lem} A transitive relation $\sim$ is a half-congruence on a \semiring0 if it is closed under addition and multiplication by the diagonal, i.e., if it satisfies the following
conditions for all $a_1,$ $a_2,$ and $b$:

\begin{equation}\label{cong11}
\begin{array}{lll}
a_1 \sim a_2   \quad  & \text{implies}  \quad &
a_1 +b \sim    a_2+b; \\[2mm]
a_1 \sim a_2   \quad  & \text{implies}  \quad &
a_1  b\sim   a_2 b. \end{array}  \end{equation}
 \end{lem}
\begin{proof} $a_1+ b_1 \sim a_2+b_1 = b_1 + a_2 \sim b_2+a_2 = a_2+b_2.$ Likewise for multiplication.
\end{proof}

\section{The layered structure}\label{Sec:SupCat}

We are ready to bring in the leading players in this theory,
 taking into account a 0-layer. \begin{defn} A pre-order is \textbf{directed} if for any $a,b$ there is $c$ such that $c\ge a$ and $c\ge b.$\end{defn}
 We assume throughout that the sorting set $L$
is a directed, (non-negative) pre-ordered  semiring \semiring0 with zero
element $0 := 0_L$; the bulk of our applications in this paper are
for $L$ ordered.  Let $ \Lstr := L \setminus \{ 0\}$. We recall
\cite[Construction~3.2]{IzhakianKnebuschRowen2009Refined}.

\begin{construction}\label{defn50} Suppose $\tG$ is a given cancellative
monoid. $R := \scrR( \Lstr,\tG)$ is   defined
 set-theoretically as $ \Lstr \times \tG $, where
 we denote the element $(\ell,a)$ as $\xl{a}{\ell}$  and, for
$ k,\ell\in L,$ $a,b\in\tG,$ we define multiplication
componentwise, i.e.,
\begin{equation}\label{13}   \xl{a}{k} \cdot \xl{b}{\ell} =
\xl{ab}{k\ell}.
\end{equation}

Addition is given by the  rules:

\begin{equation}\label{14}
 \xl{a}{k} + \xl{b}{\ell}=\begin{cases}  \xl{a}{k}& \quad\text{if }\ a >
 b,\\ \xl{b}{\ell}& \quad\text{if }\ a <  b,\\
 \xl{a}{k+\ell}& \quad\text{if }\ a= b.\end{cases}\end{equation}

 We define
  $R_\ell : = \{ \ell \} \times \tG $, for each $\ell \in  \Lstr.$
 Namely $R = \dot \bigcup_{\ell \in L } R_\ell$.
\end{construction}

This is our prototype of a layered \predomain0, and should be
borne in mind throughout the sequel. Note that in this case $R_1$
is a monoid, which is isomorphic to $\tG$.

 Nevertheless, we also consider the possibility that the monoid $\tG$ is noncancellative, in which case,
 as noted in \cite{IzhakianKnebuschRowen2009Refined}, Construction \ref{defn50} fails to satisfy distributivity
 and thus is not a semiring.

\begin{defn}\label{can3} Suppose  $\tG$ is  an
 ordered Abelian monoid.
An element $z \in \tG$ is a \textbf{noncancellative product} if $z
= ab=
ac$ for suitable $a,b,c$ with $b \ne c.$

More generally, when $(\tM, \tG, v)$ is a triple, an element $z
\in \tM$ is a $v$-\textbf{noncancellative product} if $v(z) =
v(ab)= v(ac)$ for suitable $a,b,c,$ where $v(b)\ne v(c).$

\end{defn}

\begin{prop}\label{7.2}
 The set  $A$  of $v$-noncancellative products comprises a monoid ideal of $\tM$.
\end{prop}
\begin{proof} If $v(z) = v(ab)=
 v(ac) \in  A$, then $  v(ad)v(c) =v(ac)v(d) =v(zd) = v(abd)= v(ad)v(b) $.
\end{proof}

\begin{construction}\label{defn5}  Suppose $(\tM, \cdot \, , \ge, \one_\tG )$ is an Abelian monoid, with an m-valuation
$v:\tM \to \tG,$ and $\mfa$ is a monoid ideal of $\tM$ containing
all $v$-noncancellative products.  $R := \scrR(L,\tM)_\mfa$
is defined
 set-theoretically as $(\Lstr \times (\tM \setminus \mfa
))\cup (\{0 \}\times \mfa)$, where
 we denote the element $(\ell,a)$ as $\xl{a}{\ell}$  and, for
$ k,\ell\in L,$ $a,b\in\tM,$   multiplication is  defined
componentwise, i.e., via the rules:
\begin{equation}\label{13}   \xl{a}{k} \cdot \xl{b}{\ell} =
\begin{cases}
\xl{ab}{k\ell} & \quad\text{if }  \  ab\notin \mfa,\\ \xl{ab}{0} &
\quad\text{if } \  ab\in \mfa.\end{cases}
\end{equation}
Addition is given as in Construction~\ref{defn50}.

\begin{equation}\label{14}
 \xl{a}{k} + \xl{b}{\ell}=\begin{cases}  \xl{a}{k}& \quad\text{if }\ v(a) >
 v(b),\\ \xl{b}{\ell}& \quad\text{if }\ v(a) <  v(b),\\
 \xl{a}{k+\ell}& \quad\text{if }\ v(a)= v(b).\end{cases}\end{equation}
$R_0: =  \{0 \}\times \mfa$   and
  $R_\ell : = \{ \ell \} \times (\tM \setminus \mfa) $, for each $\ell \in \Lstr.$
 Thus, $R = \dot \bigcup_{\ell \in L } R_\ell$.
\end{construction}

This encompasses the case where $\tM = \tG$ is an ordered monoid
and $v$ is the identity map. We usually refer to this special case,
in the interest of clarity.

\begin{thm}\label{maxplus23}
 $R: = \scrR(L,
 \tM)_\mfa$ is a  \semiring0, while  $R$ is a semiring iff the monoid $\tM$ is pointed, in which case
$\zero_R =  \xl{\zero_\tM}{0}.$

$R \setminus R_0$ is a \semiring0 iff $\mfa$ is  prime as a monoid
ideal of $\tM$. 

\end{thm} \begin{proof} The
verification that  $\R $ is a \semiring0 was
essentially done in
\cite[Proposition~3.3]{IzhakianKnebuschRowen2009Refined}. The
trickiest part
 again is to   verify the distributivity law
$$x(y+z)=xy+xz.$$ Write $x =  \xl{a}{k},$ $y =  \xl{b}{\ell},$ and $z
=  \xl{c}{m},$ and assume that $v(b)\ge v(c).$  If $v(ab)> v(ac),$ then clearly $v(b)> v(c)$, and
$$x(y+z) = xy =  xy + xz .$$
Thus we are done unless $v(ab) = v(ac).$

If $v(b)=v(c)$ with $ab \notin \mfa$, then
$$x(y+z) = \xl{a}{k}(\xl{b}{\ell}+ \xl{b}{m}) = \xl{a}{k}\xl{b}{\ell +m}=  \xl{(ab)}{k\ell+km} =
\xl{(ab)}{k\ell} + \xl{(ab)}{km}  = xy + xz.$$

If $v(b)=v(c)$ with $ab \in \mfa$, then
$$x(y+z) = \xl{a}{k}(\xl{b}{\ell}+ \xl{b}{m}) = \xl{a}{k}\xl{b}{\ell +m}=  \xl{(ab)}{0} =
\xl{(ab)}{0} + \xl{(ab)}{0}  = xy + xz.$$

If $v(b)>v(c)$, then $ab, ac \in \mfa,$ so
$$x(y+z) = \xl{a}{k}(\xl{b}{\ell}+ \xl{c}{m}) = \xl{a}{k}\xl{b}{\ell}=  \xl{(ab)}{0} =
\xl{(ab)}{0} + \xl{(ac)}{0}  = xy + xz.$$

When $\tM$ is pointed, the verification of the zero element is an
easy computation.

The next assertion is clear: $xy \in \mfa$ iff $xy \in R_0.$ 
\end{proof}

We have the maps $\nu_{\ell,k}: R_k\to R_\ell$ given by
$\nu_{\ell,k}(\xl{a}{k})= \xl{a}{\ell}$ for any $0<k \leq \ell$,
and a \textbf{sorting map} $ \lv: \R\to L$ given by
$\lv(\ell,a)=\ell,$ for any $a\in\tM,$ $\ell \in L.$

Note that $R\setminus \mfa$ could be a finite set, in which case
we could apply various arithmetic tools such as zeta
functions.

\begin{rem}  If $\tM = \tG$ and $\mfa = \emptyset,$ then  $R_0 = \emptyset$, and
$\scrR(L, \tG)$ coincides with the \semiring0  $\scrR(\Lstr , \tG)$ of Construction \ref{defn50}.
\end{rem}

\begin{lem}\label{mon0}  For any multiplicative idempotent $\ell$  of $L$, the subset $R_0 \cup R_\ell$ of   $\scrR(L,
\tG)$ is a monoid, together with a natural homomorphism  to $\tG.$
\end{lem}
\begin{proof} If $\xl{a }{\ell},\xl{b }{\ell} \in R_1$, then their product either is
$\xl{(a b)} {\ell^2} = \xl{(a b)} {\ell}\in  R_\ell,$ or  $\xl{(a
b)} {0} $ if $ab \in \mfa.$  The natural homomorphism is given by
$\xl{a }{\ell} \mapsto v(a).$
\end{proof}

The main application of this lemma is for $\ell =1$.  The layer
$R_1$ is of particular importance, since its   unit element
is~$\rone$. Two other obvious multiplicative idempotents of $L$
are    0 and $\infty$ (when appropriate, since $\infty$ need not
belong to $L$).

\section{Layered \semirings0}\label{mor1}

In this section we  provide the framework  for
Construction~\ref{defn5} and truncation (Example~\ref{trun0001}).
We deal with a zero layer, i.e., assume that $0
\in L,$ and  treat the zero component $R_0$ specially, taking the
opportunity to fit the zero element of $R$ (if it exists) into the
theory.   Since
we also want to consider monoids that are not cancellative, we
need to work harder to obtain distributivity. We axiomatize in
order to place the theory in a categorical framework.

\begin{defn}\label{defn10} Suppose  $(\Lz, \ge)$ is a partially pre-ordered, directed
semiring. An $\Lz$-\textbf{layered \semiring0} $$\R :=\ldsR, \qquad
$$ is a \semiring0 $\R $, together with a family $\{ R_\ell :
\ell \in L\}$ of disjoint subsets $R_\ell \subset R$,  such that
\begin{equation}\label{unionp} \R := \dot \bigcup_{\ell\in \Lz}\R
_\ell,\end{equation}  and a family
  of \textbf{sort transition maps}
$$   \nu_{m,\ell}:\R _\ell\to \R _m,\quad
\forall m\ge \ell >0 ,$$  such that $$\nu_{\ell,\ell}=\id_{\R
_\ell}$$ for every $\ell\in \Lz,$ and
$$\nu_{m,\ell}\circ \nu_{\ell,k}=\nu_{m,k} , \qquad  \forall m \geq \ell \geq k, $$ whenever both sides
are defined. To avoid complications, we assume that any element of
$R_0$ can be written as a product $ab$ where $a,b \in R\setminus
R_0$.  We also require the axioms A1--A4, and B, given presently,
to be satisfied.pave
(In order to have our definition compatible with the $\Lz$-layered
\predomains0\ of \cite{IzhakianKnebuschRowen2009Refined}, we
permit $R_0  = \emptyset$.)

We also require   $R_\infty$  to be the direct limit of the $R
_\ell,$ $ \ell >0 $, together with maps $\nu_{\infty,\ell}: R
_\ell \to R _\infty$, which extend to a map   $\nu:R \to R
_\infty$. (For $c = ab \in R_0$ we define $ \nu(c)
=\nu(a)\nu(b).$)

 We write $a ^\nu$ for $\nu(a)$.
 We  write $a \nucong b$ for $a \in R_k$ and $b \in R_\ell$  whenever
 $a ^\nu = b ^\nu$, which means $\nu_{m,k}(a) = \nu_{m,\ell}( b )$ in $R_m$ for some $m \ge
 k,\ell$. (This notation is used generically: We write  $a \nucong b$
even when the sort transmission maps  are denoted differently.)

 Similarly,    we write  $a \le _\nu b$ if  $a ^\nu + b^ \nu= b^ \nu$,
 which means $\nu_{m,k}(a) + \nu_{m,\ell}( b )=  \nu_{m,\ell}( b )$ in $R_m$ for some $m \ge
 k,\ell$, and we write $a < _\nu b$ if $a \le _\nu b$ but  $a \not \nucong b$.

 The axioms are as follows:

\boxtext{
\begin{enumerate}

\item[A1.] $\rone \in \R _{1}.$ \pSkip

 \item[A2.] If $a\in \R _k$ and  $b\in
\R _\ell,$ then $ab\in   \R _ {k \ell} \cup R_0 $. \pSkip

\item[A3.] The product in $\R $ is compatible with sort transition
maps: Suppose $a\in \R _k,$ $b\in \R _{\ell},$ with $m\ge k$ and
$m'\ge \ell.$ Then
$$\nu_{m,k}(a)\cdot\nu_{m',\ell}(b)= \nu_{mm',k\ell}(ab).$$

\item[A4.] $\nu_{\ell,k}(a) + \nu_{\ell',k}(a)
 =\nu_{\ell+\ell',k}(a)  $ for all $a \in R_k$ and all $\ell, \ell' \ge k.$

\item[A5.] If $a \in R_k$, $b \in R_{\ell}$,
 and $c = a+b \in R_{k'}$, then $$\nu_{m,k'}(c) = \nu_{m,k}(a) +
 \nu_{m,\ell}(b)$$ for each $m \ge k+\ell$.

  \item[A6.] $R_0$ is an additive semigroup (and thus an ideal) of
$R$.

\end{enumerate}}

 \boxtext{
\begin{enumerate}

\item[B.] (Supertropicality) Suppose  $a\in \R _k,$ $b\in \R
_{\ell},$ and $a \nucong b$. Then \\ $a+b \in R_{k+\ell}$ with
$a+b \nucong a$.  If moreover $k = \infty$, then $a+b =
a.$
\end{enumerate}}

We say that any element $a$ of $\R _k$ has \bfem{sort}~$k$ $(k\in
L)$. $L$ is called the \textbf{sorting semiring} of the layered
 \semiring0 $R = \bigcup_{\ell \in L} R_\ell$. Thus, $\ell$ is
the \textbf{sort} of the \textbf{layer} $ R_\ell$.

The \textbf{sorting map} $s:\R\to L$ is the map that sends every
element $a\in\R_\ell$ to its sort $\ell$.

(Taken from \cite[Definition~5.2]{IKR4}) An $\Lz$-layered
\textbf{pre-\domain0} is an $\Lz$-layered \semiring0 in which
Axiom A2 is strengthened to the condition $ab\in \R _{k \ell }.$
 An $\Lz$-layered \semiring0 $\R := \ldsR$ is  called \textbf{uniform}
 when the sorting \semiring0 $L$ is totally ordered
and the sort transition maps $\nu_{\ell,k}$ are bijective for each
$\ell > k > 0.$

\end{defn}


%

\begin{defn}
 An \textbf{$\Lz$-layered \predomain0} is an $\Lz$-layered
\semiring0 $R$ for which   $R_1$ is a monoid.
 \end{defn}

%

\begin{defn}


An $\Lz$-layered \semiring0 is $\nu$-\textbf{bipotent} if $a+b \in
\{a, b\}$ whenever $a\nnucong b.$

 An $\Lz$-\textbf{layered \bidomain0} is
a $\nu$-bipotent $\Lz$-layered \domain0.

\end{defn}

\begin{rem}\label{bip1}
For layered \bidomains0, Axiom A5 says that $a < _\nu b$ implies
$\nu_{m,k} (a) < _\nu \nu_{m,\ell} (b)$.\end{rem}

Let us put Construction \ref{defn5} into context, using the
layered version of Definition \ref{can3}. An element $z \in R$ is
a $\nu$-\textbf{noncancellative product} if $z^\nu = a^\nu b^\nu=
a^\nu c^\nu$ for suitable $a,b,c,$ where $b\not \nucong c.$ Note
that the set of $\nu$-noncancellative products of an $\Lz$-layered
\semiring0 is an ideal. The potential for noncancellative products
was one motivation for  Construction \ref{defn5}, so the next
result becomes relevant.

%
%


\begin{prop}\label{can1} Suppose $z=$ is a
$\nu$-noncancellative product, with $\ell = s(z)$. Then  $\ell =
 2\ell.$ In particular, if  $\ell$ is
finite, then $\ell = 0.$
\end{prop}
\begin{proof} If $\xl{z}{\ell}
= ab \nucong ac$ with $b^\nu > c^\nu,$ then
$$\xl{z}{\ell} = ab = a(b+c) = ab +ac = \xl{z}{\ell} + \xl{z}{\ell} = \xl{z}{2\ell},$$ implying $\ell = 2\ell .$
\end{proof}

Since   0 and
$\infty$ are multiplicative idempotents of $L$, one could formulate an
analogous definition using the layer at $\infty$ instead of at 0, and
 indeed this version is implicit in some of our work on superalgebras and supervaluations, such as \cite{IKR5} and \cite{IzhakianRowen2007SuperTropical}. However, there are several good reasons for using the 0 layer in place of the $\infty$ layer.

\begin{enumerate}

  \item $R_\infty$ corresponds to the image of the ghost map $\nu$, which may involve considerable contraction. On the other hand, we often do not want any contraction to $R_0.$

\item In some ways,   $R_0$ and $R_\infty$ should be complements,
as indicated presently.

\item  $R_0$ is an ideal which behaves much like a zero element. In particular, it is more intuitive for the zero element (if it exists) to belong to   $R_0$.

 \item Remark \ref{tan0} below formalizes the notion that $R_0$ also has tangible properties.

\end{enumerate}

 \begin{rem}\label{inf0} The 0-layer   and the  $\infty$-layer behave  similarly, since both $0$ and $\infty$ are
absorbing elements of $L$, except that $0$ also absorbs $\infty$
in the sense that $0 \cdot \infty = 0.$ In case $\infty \in L$ but $0 \notin L$,
$R_\infty$ is an ideal of $R$ that can often be used to replace
$R_0$ in the above discussion.

One difference between the 0 layer and the $\infty$
layer is that for $a \nucong b$ with $b\in R_\ell,$  if $a \in
R_0$   then $s(a+b) = \ell,$ whereas if $a \in R_\infty$ then
$s(a+b) = \infty.$
\end{rem}

\begin{lem}\label{esot}
 The layer $R_0$ is also
 an ideal of~$R$.
If furthermore $\rzero \in R$,  then $\rzero \in R_0$.
\end{lem}\begin{proof} The first assertion is clear. Suppose $\rzero \in R_k.$  Then for any $a \in
R_0$ we have
$$\rzero  = \rzero \cdot a \in R_{k\cdot 0} = R_0.$$ \end{proof}

\begin{rem}
 If $\infty \in L$, then
$R_\infty $ is a monoid, and $R_0 \cup R_\infty $ is an ideal of
$R$. \end{rem}

\begin{lem}\label{gen1} If $\tM$ is any submonoid of a layered \semiring0  $\R
:=\ldsR,$ then the additive sub-semigroup $\overline{\tM}$ of $R$
generated by $\tM$ is also a layered \semiring0.\end{lem}
 \begin{proof} $\overline{\tM}$ is closed under multiplication, and thus is a \semiring0. Axiom  A1 is given, and the other axioms follow
 a fortiori.
 \end{proof}

\subsection{The \ZO-submonoid}

Since in general $R_1$ no longer turns out to be a monoid, we must
also take into account the 0-layer.

\begin{rem}\label{tan0}
$ R_0 \cup R_1$ is a submonoid of $R$.
\end{rem}

\begin{defn}\label{ghostsur} The
\textbf{\ZO-submonoid} is the submonoid of $R$ generated by $R_1$.
\end{defn}

Thus, the \ZO-submonoid is contained in $ R_0 \cup R_1$. Since
$\rone \in R_1,$ every invertible element of the fundamental
submonoid must lie in $R_1$.

\begin{prop}\label{maxplus2}  Suppose $R$ is an L-layered \semiring0. Then $\nucong$ is an equivalence relation,
whose  set  $\tG$  of equivalence classes is a monoid, which is
ordered when $R$ is $\nu$-bipotent. In this case, the
 \ZO-submonoid  $\tT$   of $R$ has an m-valuation $\nu : \tT \to
 \tG$ satisfying $a \mapsto [a^\nu].$
\end{prop}
\begin{proof} $\nucong$ is an equivalence relation by definition,
and the equivalence classes comprise a monoid in view of Axiom A3.
When $R$ is $\nu$-bipotent, we get an ordered monoid by Remark~
\ref{bip1}, and $\nu$ is an m-valuation by Axiom A3.
\end{proof}
%

  We are interested in generation by the
\ZO-submonoid.

\begin{defn}\label{tan000}
The \textbf{tangibly generated} sub-\semiring0 $R_{\langle 1
\rangle}$ of an $L$-layered \semiring0 $R$
 is the sub-\semiring0 generated by $R_1$; the \semiring0 $R$ is
\textbf{tangibly generated} if $R_{\langle 1 \rangle}=
R$.\end{defn}

Thus, $R$ is  tangibly generated if  $R = \left(\bigcup _{k\in L}
\nu_{k,1}(R_1)\right)\cup R_0$.
Passing to $R_{\langle 1 \rangle}$ may shrink the sorting set.

\begin{lem}\label{exp0} The tangibly generated sub-\semiring0 $R_{\langle 1
\rangle}$ of a $\nu$-bipotent layered \semiring0  is a tangibly
generated, $\nu$-bipotent layered \semiring0 with respect to the
sorting sub-\semiring0 of $L$ generated by $1_L$. If~$R$ is a
layered pre-\domain0, then $R_{\langle 1 \rangle}$ is a layered
pre-\domain0 whose 0-layer is empty.
\end{lem}
\begin{proof} The axioms are verified a fortiori,
since addition only involves adding sorts, starting with $1_L$. For the second assertion, since addition cannot lower the sort, we do not get any elements of sort 0.
\end{proof}

Thus, replacing $R$ by its tangibly generated sub-\semiring0
enables us to assume that $(L,+)$ is generated by 1 and 0.
%

\begin{example}\label{tan21} Construction \ref{defn5} is tangibly
generated.
 \end{example}

It turns out that we could develop the theory under the weaker
condition that $L$ is a partially pre-ordered multiplicative
monoid, and we sketch the appropriate changes at the end of the ~Appendix.

\begin{example}\label{moveideal} Given any ideal $\mfa$ of an $L$-layered \semiring0 $R$, we formally define
$R_\mfa$ to be  $R$  with  the same \semiring0 operations, and to
have  the same sort function as $R$, except that now $s(a) = 0$
for every $a \in \mfa.$  In other words,
$$(R_\mfa)_0 : = R_0 \cup \mfa; \qquad (R_\mfa)_\ell := R_\ell
\setminus (\mfa \cap R_\ell ).$$

Now define $$\bar\mfa := \{ b \in R : b \nucong a \}.$$ Then $\bar
\mfa \triangleleft R, $ so we could use $\bar \mfa $ instead of
$\mfa .$
\end{example}

 \begin{prop}\label{zerolay}  $R_\mfa$ is a \semiring0.\end{prop}  \begin{proof} We need to check associativity and distributivity. But
this is clear unless we are using elements of~$\mfa,$ and then
associativity holds because all products have layer 0. Likewise,
to see that $a(b+c)$ and $ab + ac$ have the same layer, note this
is clear if $s(a) = 0$ or if $s(b+c) \ne 0.$ Thus we may assume
that  $s(b+c) =0$, and again we are done if $s(b) = s(c) = 0,$ so
we may assume that $s(b) =0$ and $s(c)\ne  0$ with $ b >_\nu c$
but $ab \nucong ac.$ But then
$$s(a(b+c)) = s(ab) + s(ac) = s(ab),$$
so $a(b+c) = ab = ab+ac.$\end{proof}
\begin{rem}\label{finitear} If $R \setminus \mfa$ is finite, then  $(R_\mfa)_1$ is
a finite set. Thus, we have a way of ``contracting'' the tangible
component to a finite set.
\end{rem}

 One  instance  of arithmetic
significance is when $R  = \scrR(L,\Net \cup \{0\})$ where  $\mfa = \{
\xl{n}{\ell} : n
> q ,\, \ell \in L\}$ for some $q \in \Net.$ In this case, we can ``compress''  $\mfa$ to a single
element in $R_0$.

\begin{example}[The layered $\nu$-truncated \semiring0]\label{trun0001} Take an ordered semiring $L $ and ordered
triple $(\tM, \tG, v)$, with $R = \scrR(L\setminus \{ 0
\},\tM),$ and fixing $q
>\one_\tM$ in $\tM$, define $$\mfa : = \{\xl{a}{\ell} : v(a) \ge v(q)\} \triangleleft R.$$
Then $R_\mfa$ contracts to the $L$-layered \semiring0 $$ \big \{
\xl{a}{k}: \ell \in L, \ a < q \big \} \cup \{ \xl{q}{0} \},
$$ where addition is defined as in Construction~\ref{defn5}, and
multiplication $ \xl{a}{k} \xl{b}{\ell} $ is given as in
Equation~\eqref{13} except for $ab \ge q$, in which case
$\xl{a}{k} \xl{b}{\ell} =
  \xl{q}{0}$ for any $k,\ell \in L.$ Addition is given by
$$\xl{a}{k} +   \xl{q}{0} =   \xl{q}{0}.$$

The sort transition maps are as in  Construction~\ref{defn5}.
Thus, $\xl{q}{0}$ is the special infinite
element.

When instead the layering \semiring0 $L$ is finite, we see that
$R_1 \cup \{\xl{q}{0}\}$ is a finite set, which merits further
study using arithmetic techniques.
\end{example}

Here is a way to make $L$ finite.

\begin{example}[The $L$-truncated \semiring0]\label{trun0002} Take an ordered semiring $L $ and ordered
triple $(\tM, \tG, v)$, with $R = \scrR(L\setminus \{ 0
\},\tM),$ and fix $m
>\one_\tM$ in $L$.
Then $R_\mfa$ contracts to the $L$-layered \semiring0 $$ \big \{
\xl{a}{\ell}: \ell \le m \in L,\ a \in \tM \big \},
$$ where addition is defined as in Construction~\ref{defn5}, and
multiplication $ \xl{a}{k} \xl{b}{\ell} $ is given as in
Equation~\eqref{13} except for $k\ell \ge m$, in which case
$\xl{a}{k} \xl{b}{\ell} =
  \xl{ab}{m}$. Addition is given by
$$\xl{a}{k} +   \xl{\zero}{m} = \xl{a}{k}.$$
The sort transition maps are as in  Construction~\ref{defn5}.
Thus, $R_{m}$ is the special infinite
layer. 
\end{example}

Thus, the two kinds of truncation can interweave to create finite
layered structures.

\subsection{The case of onto sort transition maps}


We write $e_\ell$ for $\nu _{\ell,1} (\rone).$ Here is a key
simplification for layered \domains0 when the sort transition maps
are onto, which enables us to reduce many results to the tangible
case:
 \begin{lem}\label{Krems1} If $R$ is an $L$-layered \semiring0 and $a
 \in R_\ell$ with $\nu_{\ell,1}:  R_1 \to  R_\ell$ onto, then $a = e_\ell a_1$ for some $a_1 \in R_1$.
\end{lem}
\begin{proof}
Taking $a_1 \in R_1$ for which $\nu_{\ell,1}(a_1) = a,$ we have $a
= \nu_{\ell,1} (a_1) =  e_\ell a_1$.
\end{proof}

\begin{Note}\label{uniforme} Lemma~\ref{Krems1} enables us to
simplify the theory for any layer $\ell>1$ for which
$\nu_{\ell,1}$ is onto. When $\ell < 1$ we could go in the
opposite direction, and define $e_\ell$ such that
$\nu_{1,\ell}(e_\ell) = \rone.$ This will be well-defined when
$\nu_{1,\ell}$ is 1:1 since, writing $\ell = \frac mn$ for any
$a\in R_\ell$ with $\nu_{1,\ell}(a) = \rone,$ we have
\begin{equation}\label{compe}  n e_{\ell}  = n e_{m/n}  =
 e_{m} = \nu_{m,\ell}(a) =  n a,  \end{equation} implying $a =  e_{\ell}.$

\end{Note}



\subsection{Adjoining the 0-layer}\label{adj0}

 Starting with an $L$-layered \predomain0 $R$ with respect to a \semiring0 $L$, we can adjoin
a zero layer $R_0$ formally in several ways. The first way is
simply by adjoining a zero element to~$R$.

\begin{rem}\label{adjoin00}

For any  layered \predomain0 $R$  with respect to a \semiring0
${\Lz}$, the  semiring
$$\R  \ \dot\cup\  \{ \rzero\}$$ can be
 layered with respect to the semiring $$  \zL := \Lz \ \dot\cup\
\{ \lzero \},$$ 
where we take   $R_0 :=\{\rzero \} $, putting it in the zero layer
as seen by applying the argument of Proposition~\ref{zerolay}. We
take the sort transition maps $\nu _{0, \ell} (a): = \rzero$ for
all $\ell \ne 0$ and $a \in R.$
\end{rem}

However, this is not the only possibility for the zero layer, as
we saw in \cite[Remark~3.8]{IzhakianKnebuschRowen2009Refined}.

\begin{construction}\label{wholelay} If $R$ is a uniform $L$-layered \predomain0, where $L$ is
a \semiring0,  then adjoining~$\{ 0 \}$ formally to~ $L$ as the
unique minimal element,  we can form a uniform $\zL $-layered
\semiring0 $R \cup R_0,$ where $R_0 := e_0 R_1$ is another copy of
$R_1,$ under the same rules of addition and multiplication given
by Construction~\ref{defn50}.
\end{construction}
\begin{proof} If $a = e_0 a_1$,  $b= e_k b_1,$ and $c = e_\ell c_1  $ for $a_1, b_1, c_1 \in R_1,$ then
$$(ab)c  = e_0 e_k e_\ell (a_1b_1) c_1  = e_0 e_k e_\ell a_1(b_1 c_1) = e_0   a_1(b_1c_1)  = e_0 a_1(b_1 c_1) =
a(bc),$$ yielding associativity of multiplication. To see
distributivity, we note that $ e_k b_1+ e_\ell c_1 = e_m (b_1 +
c_1)$ where $m \in \{k, \ell, k+ \ell\},$ so
$$a(b+c) = e_0 e_m  a_1(b_1+c_1) = e_0
a_1(b_1+c_1)  = e_0 a_1b_1 + e_0 a_1c_1 = e_0 e_k a_1b_1 + e_0
e_\ell a_1c_1 = ab+ac.$$

Associativity of addition is similar. Finally, if $a = \rzero  \in
R_0$ and $b\in R_\ell$, then    $ab \in R_{0 \cdot \ell}= R_0.$
\end{proof}

 Since we  have several ways of adjoining a zero layer, the
following observation is useful.

\begin{prop}    For any semiring $R$ layered with respect to a
\semiring0 ${L} ,$   $ \R \ds{\dot\cup} \{ \rzero\}$ is an
$\zL$-layered sub-semiring of $R \ds{\dot\cup} R_0.$

More generally, for any ideal $\mfa$ of $R$, writing $\mfa_0 $ for
$\mfa \cap R_0 $, we have $(\bigcup_{\ell \ne 0}\ \R_\ell)
\ds{\dot\cup} \mfa_0$ is an $\zL$-layered sub-semiring of $R
\ds{\dot\cup} R_0.$
\end{prop}

\begin{proof} If $a   \in
\mfa_0$ and $b\in R_\ell$, then    $ab  \in R_{0 \cdot \ell}= R_0,$
implying $ab \in \mfa_0$.
\end{proof}

This gives rise to the question of whether we should adjoin the
entire 0-layer, or just $\rzero?$ Although one's experience from
classical algebra might lead one to adjoin only $\rzero,$ there
are situations in which one might need other elements in $R_0$ in
order to distinguish polynomials.

\subsection{Adjoining the absolute ghost layer, and the passage to standard supertropical
\domains0}\label{adjoininf}

 Even when $L$ originally does not contain an infinite element a priori,   $L$-layered \bidomains0  tie in directly with the
(standard) supertropical theory, via a  ghost layer introduced at
a new element $\infty$ which we adjoin. (This works even when
$(\ge)$ is merely a partial order on $L$, although it it is easier
when $(\ge)$ is a total order.)

\begin{rem}\label{directlim} Any $L$-layered \semiring0  $\ldsR$ is a directed system
with respect to the set $L$, as described in  \cite[p.~71]{Jac}.
Hence, by \cite[Theorem 2.8]{Jac}, the layers $R_k$ have a direct
limit which we denote $R_{\infty},$ and maps $$\nu_{\infty,k} :
R_k \to R_{\infty}$$ such that $\nu_{\infty,k} = \nu_{\infty,\ell}
\circ \nu_{\ell,k}$ for each $a \in R_k$ and all $k<\ell$. Since
$R = \bigcup_k R_k,$ we can piece together these maps
$\nu_{\infty,k}$ to a map $\nu: R \to R_{\infty}.$ We define
\begin{equation}\label{1} e = e_\infty :=\nu(\rone),\end{equation}
easily seen to be the unit element of  $R_{\infty}.$

We write $a^\nu$ for $\nu(a) \in R_\infty$. Thus $a^\nu  = b^\nu$
iff $a \nucong b$ in our previous notation.
\end{rem}

We call $R_\infty$ the \textbf{absolute ghost layer} and  $\nu$
the (absolute) \textbf{ghost map} of $\R.$ Note that in the
uniform case, $R_\infty$ is just another copy of $R_1,$ so
we can dispense with direct limits.

\begin{thm}\label{Udef} Suppose $\R := \ldsR$ is an $L$-layered \semiring0. Then the absolute ghost layer   $R_{\infty}$ is
a bipotent \semiring0. The ghost map $\nu: R \to R_{\infty} $ is a
\semiring0\ homomorphism. Define $$U = U(R): =  R \ \dot\cup\ \
R_{\infty}.$$ Then $U$ is a \semiring0 under the given operations
of $R$ and $R_\infty,$ together with
$$
\begin{array}{rll}
a\cdot b^\nu & := & (ab)^\nu; \\[2mm]
 a +b^\nu & := & \begin{cases} a&\quad\text{if \ $ea>eb$,}\\
b^\nu &\quad\text{if \  $ea\le eb$.} \end{cases}\end{array}
 $$
 Also, extend $\nu$ to a map $\nu_U : U \to R_{\infty}$ by taking
$\nu_U$ to be the identity on $R_{\infty}.$ Then $U$ has ghost
ideal $\tG = \tG(U) := R_{\infty},$ in the sense of
\cite{IzhakianRowen2007SuperTropical}, and $\nu_U(a) = ea$ for
every $a$ in $U .$

Then $U$  can be modified to a supertropical \semiring0 $$\cR
_{1,\infty} := R_1 \ \dot\cup \ \tG,$$ retaining the given
multiplication $\cdot$ of $U,$ but with  new addition $\oplus$
given by the rules
\begin{equation}\label{6}
a\oplus b:=\begin{cases} a&\quad\text{if \ $ea>eb$,}\\
b &\quad\text{if \ $ea<eb$,}\\
ea &\quad\text{if } \  ea=eb. \end{cases}\end{equation}

\end{thm}


\begin{proof}
Axiom A3  tells us that
$$\nu_{mm',k\ell}(a \cdot b)=\nu_{m,k}(a) \cdot \nu_{m',\ell}(b)$$
for any $a\in R_k$ and $b\in\R_\ell;$ taking limits yields $$\nu(a
\cdot b)=\nu(a) \cdot \nu(b).$$ Likewise, Axiom  B tells us that
$$\nu(a+b)=\nu(a)+\nu(b).$$
 The other verifications are also easy.  By
\eqref{1} we have
\begin{equation*}\label{2}
\nu(x)=e \cdot x\quad\text{for every}\quad x\in\R.\end{equation*}
Thus $\nu\circ\nu=\nu$, and also $\nu: \R\to \tG$ is a \semiring0\
homomorphism from $\R$ onto $\tG = \tG(U).$

  We  extend the $\nu$-equivalence relation from
$R$ to $U$ by decreeing that  $a \equiv_U  b$  iff $a$ and $b$
have the same value under $\nu.$

We turn to the last assertion. Due to \eqref{6} we have
\begin{equation*}\label{7}
a\oplus b=a+b\quad \text{if}\quad  a\not \nucong b.\end{equation*}
On the other hand,
\begin{equation*}\label{8}
a\oplus b=e(a+b)\quad\text{if}\quad a\nucong b.\end{equation*}
Note that
\begin{equation*}\label{9}
a\oplus b \nucong a+b \end{equation*} in all cases. Also, $\tG(U)
: = R_\infty  = \tG(\cR _{1,\infty}).$ \end{proof}


 We may regard $\cR _{1,\infty} : =
(\cR _{1,\infty},\oplus,\cdot \, )$ as a degeneration of the
\semiring0 $U := U(R),$ where all the ghost layers have been
coalesced to $R_\infty.$  When   $L = L_{\ge 1}$, then there is a
\semiring0 homomorphism $U \to \cR _{1,\infty}$
given by $$a \mapsto \begin{cases} a \quad\quad \text{ for }\quad a \in R_1 \cup R_\infty,\\
\nu(a)  \quad \text{otherwise}. \end{cases}$$

We are now in a position to see why Construction~\ref{defn50} of a
uniform $L$-pre-\domain0  is generic. We recall
\begin{equation*}\label{10} \R(L,\tG) :=\{\xl{a}{\ell} \bigm| a\in\tG,\ \ell\in L\}.\end{equation*}

\begin{rem}\label{tang100} In a uniform $\Lz$-layered \predomain0, we can define $\nu_{k,\ell}$ for $0<k< \ell$ to be
$\nu_{\ell,k}^{-1}.$ Thus, $\nu_{k,\ell}$ is defined for all
$0<k,\ell \in L.$\end{rem}

\section{Morphisms of layered \semirings0}\label{sec:Morphisms}

In ordered to understand layered categories, we need a good notion
of morphism. This is easiest to describe for layered \domains0.
\subsection{Layered homomorphisms}\label{homomorph}

 We assume that $L$ is non-negative. 


\begin{defn}\label{morph1} A \textbf{layered homomorphism} $$(\vrp, \rho ): \ldsR \ \to \ \ldsLRpr$$ of uniform $L$-layered \predomains0  is a  \semiring0\
homomorphism
  $\rho: L \to L'$ preserving the given partial orders, i.e., satisfying the condition:
\boxtext{

\begin{enumerate}
 \item[M1.] $k \le \ell $ implies $ \rho(k) \le  \rho(\ell). $

\end{enumerate}}
 together with a \semiring0
 homomorphism $\vrp: R \to R'$ such that
\boxtext{
\begin{enumerate}
    \item[M2.] $\lv '(\vrp
 (a)) \ge \rho (\lv(a)), \quad \forall a \in R. $


\end{enumerate}}
\end{defn}

The definition becomes more complicated when $0 \in L;$ then we
need to modify Axiom~M2 to: \boxtext{
 \begin{itemize} \item[M2'.] $\lv '(\vrp
 (a))  = \ell,$ where $\ell = 0$ or $\ell \ge \rho (\lv(a)), \quad \forall a \in R. $
\end{itemize}}

  We always write $\Phi := (\vrp, \rho
).$
We often assume $L = L'$ and $\rho =
\id_L;$ we call $\Phi$ a \textbf{natural homomorphism} in this
situation.

\begin{lem}\label{Phidet}  Write $e_{ \ell,R}$ for $e_\ell$ in
$R$. Then $\vrp(e_{\ell,R}) = e_{\ell,R'},$ for each $\ell$ in the
sub-\semiring0 of $L$ (resp.~$L'$) generated by $1$.
\end{lem}\begin{proof} Then  $\vrp ( e_{1,R}) = \vrp (\rone) = \one_{R'} =  e_{1,R'}.$ Thus, for each $n \in \Net,$ we have
$$\vrp(e_{n,R}) = \vrp ( e_{1,R} + \cdots + e_{1,R}) =   \vrp ( e_{1,R}) + \cdots + \vrp
(e_{1,R}) =  e_{1,R'} + \cdots + e_{1,R'} = e_{n,R'}.$$
\end{proof}

It follows at once that the homomorphism $\vrp$ is given  by its action on $R_1$.
\begin{prop}\label{multip} If $a = e_\ell a_1$ as in Lemma~\ref{Krems1}, then
\begin{equation}\label{M3}\vrp(a) =  \vrp(e_{\ell,R}) \vrp(a_1) =
e_{\ell,R'}\vrp(a_1), \qquad \forall \ell>0. \end{equation}
\end{prop}
\begin{proof} $\vrp(a) =  \vrp(e_{\ell,R}) \vrp(a_1) =
e_{\ell,R'}\vrp(a_1).$
\end{proof}

\begin{cor} Equation~\eqref{M3} holds automatically whenever $R$ is uniform
$L$-layered.
\end{cor}
\begin{proof}  Lemma~\ref{Krems1} is applicable.
\end{proof}


%
%

\begin{prop}\label{tan2} Suppose   $\varphi: R   \to R'$ is a layered
homomorphism, and  $R$ is tangibly generated. Then $\varphi$ is
determined by its action on $ R_0 \cup R_1$, via the formula
$$ \varphi \bigg(\sum_i a_i\bigg) =  \sum_i  \varphi(  a_i) .$$
\end{prop}
\begin{proof} It is enough to check sums, in view of
Lemma \ref{gen1}.
\end{proof}

Our category will be comprised of the tangibly generated
$L$-layered \semirings0.

\subsection{Layered supervaluations and the layered analytification}

In case our layered \semiring0 is not uniform, we need a more
general notion of morphism, treated in \cite{IKR4}. To understand
what is going on, we need to generalize the notion of
``valuation.'' Valuations are important in algebraic geometry, and
play a key role in tropical theory largely because of the
following example.
\begin{example}\label{analyt0}
 Recall that  the field $\mathbb K$   of
\emph{Puiseux series} over an algebraically closed field $F$ is defined to be the set of all series  of the form
\begin{equation}\label{eq:PowerSeries}
 p(t) :=
\sum_{\tau \in T} c_{\tau} t ^{\tau},\qquad c_\tau \in F,
\end{equation} with $T \subset \Rati$  well-ordered and bounded from below, endowed with the valuation  $ \nVal : \Fld  ^\times \To
 \Real   \ $ given by
\begin{equation}\label{eq:valPowerSeries}
  \nVal(p(t))\  :=
    - \min \{\tau \in T \ : \; c_{\tau} \neq \zero_F \}, \quad p(t) \in
    \Fld\setminus \{ \zero_{\Fld} \}.
\end{equation}
\end{example}

A word about notation: Given a valuation (or, more generally, an m-valuation) $v$, one can
replace $v$ by $-v$ and reverse the customary inequality to get
$$v(a+b) \le \max\{v(a), v(b) \},$$
which is more compatible with the max-plus set-up.

Payne \cite{Pay2} has developed an algebraic version of
Berkovich's theory of analytification, which can be viewed as the
limit of tropicalizations. In his theory, a \textbf{multiplicative
seminorm} $|\phantom{t}|: W \to \Real$ on a ring $W$ is a
multiplicative map satisfying the triangle inequality $$|a+b| \le
|a| + |b|.$$ 
The underlying space in Payne \cite{Pay1} is the set of
multiplicative seminorms from $K[\lm_1, \dots, \lm_n]$ to
$\Real_{>0}$ extending $v$, for a given m-valuation $v: K \to
\Real_{>0}$. We generalize this definition by taking an arbitrary
ordered \semiring0 instead of $\Real_{>0}.$

The supertropical version, the \textbf{strong supervaluation}, is
defined in \cite[Definition~4.1 and
Definition~9.9]{IzhakianKnebuschRowen2009Valuation}  as a monoid
homomorphism $\vrp$ satisfying $\vrp(a)+\vrp(b) \lmodg \vrp(a+b)$,
where $\lmodg$ is the ghost surpassing relation of
\cite[Definition~9.1]{IzhakianKnebuschRowen2009Valuation}.
 In  this way, strong supervaluations generalize seminorms.

Here is the layered analog. 

\begin{defn}\label{layeredv} A   \textbf{layered supervaluation} on a ring $W$ is a map $\vrp: W\to R$
  from $W$ to an $L$-layered semiring~$R$ with the following
  properties:
 \begin{alignat*}{2}
&LV1:\ &&\vrp(1)=\rone,\\
&LV2:\ &&\forall a,b\in R: \vrp(ab)=\vrp(a)\vrp(b),\\
&LV3:\ &&\forall a,b\in R: \vrp(a+b) \le _\nu  \vrp(a)+\vrp(b) ,\\
&LV4:\ &&\vrp(0)=\rzero.
\end{alignat*}

A   $\BB$-\textbf{layered supervaluation} on a ring $W$ is a
layered supervaluation $\Phi: W ^ \times \to R$, where $W ^
\times:= W\setminus \{ 0 \}$, such that $\Phi(W) \subseteq R_0
\cup R_1.$

\end{defn}


\begin{prop} Suppose that  $R := \scrR(L, \tG)$ an  $L$-layered \bidomain0. If $\Phi: W \to \tG$
is a $\BB$-layered supervaluation of a ring $W$, then $\Phi(a)$ is
tangible for every invertible element $w$ of $W$. (In particular,
if $W$ is a field, then $\Phi(W ^ \times)$ is tangible.)
\end{prop}
\begin{proof} $\Phi(w)\Phi(w^{-1}) = \Phi(1) = \rone,$ so
$\Phi(w)\notin R_0,$ and thus is tangible.
\end{proof}

In this situation, the tangible layer determines the layered
supervaluation.
%

 The morphisms in the layered category should then
be those maps which transfer one layered supervaluation to
another. In the standard supertropical situation, these are the
transmissions of \cite{IKR3}, which are given in the layered
setting in~\cite{IKR4}.
 This paves the way for the following concept, with,
notation as in~Example~\ref{analyt0}:

\begin{rem}\label{analyt}  Let $R := \scrR(L,\tG)$, and view $\nVal$ as the
composite map of monoids $$\mathbb K \overset \nVal \to \tG \cong R_1
\subseteq R.$$ Then for any affine algebraic variety $X$ over $\mathbb K$,
 the space  of
$\BB$-{layered valuations} from $\mathbb K[\lm_1, \dots, \lm_n]$ to $R$
 extends
$\mathbb K^{\operatorname{an}}$ of \cite{Pay2}, and its theory invites
further study.
\end{rem}

\subsection{Surpassing and surpassed maps}

In line with the philosophy of this paper, we would like to
introduce the category of $L$-layered \semirings0. Having the
layered \semirings0 in hand, we next need to define the relevant
morphisms. From now on, to avoid complications, we assume that $R$ is a uniform, $L$-layered
\predomain0. As indicated in the introduction, although  the natural
definition from the context of \semirings0  is good enough for
most purposes, a sophisticated analysis   requires us to consider
the notion of ``supervaluation,'' and how this would relate to
morphisms that preserve the properties of supervaluations, which
we will discuss in \S\ref{layval}. But a more naive approach
suits our needs in many situations.

 \subsubsection{The surpassing relation}\label{ghocom1}

 For $\ell \in L$, an
$\ell$-\textbf{ghost sort} is an element of the form $\ell + k$,
for positive $k \in L$. We say $a$ is $\ell$-\textbf{ghost} if
$s(a)$ is an $\ell$-ghost sort. Note that the infinite element
$\infty$ of $L$ is a ``self-ghost sort,'' in the sense that
$\infty +m = \infty$ implies that $\infty$ is an $\infty$-ghost
sort.

Here is a key relation in the theory.

\begin{defn}\label{ghostsurpL}   The \textbf{surpassing $L$-relation}
$\lmodWL$ is given by
 \begin{equation}  a \lmodWL b \quad
\text{  iff either } \quad  \begin{cases} a=b+c &  \quad
\text{with}\quad c\quad  s(b)\text{-ghost},
  \quad \\ a=b,\\   
   a\nucong b & \quad\text{with}\quad
 a \quad  s(b)\text{-ghost}.\end{cases} \end{equation} \end{defn} It follows that if
$a \lmodWL b$, then $ a+b$ is $s(b)$-ghost. When $a \ne b$, this
means $a \ge _\nu b$ and $a$ is $s(b)$-ghost.

\begin{defn}\label{surmor1}  The \textbf{surpassing ($L,\nu )$-relation}
$\lmodWLnu$ is given by
 \begin{equation}  a \lmodWLnu b \qquad
\text{  iff } \qquad a \lmodWL  b \
\text{ and }\
   a\nucong b .\end{equation} \end{defn}

\begin{lem}\label{surmor2}  The surpassing $L$-relation
$\lmodWL$  and the surpassing ($L,\nu )$-relation $\lmodWLnu$ are
half-congruences.\end{lem}
\begin{proof} Pointwise verifications.
\end{proof}

\begin{rem} The congruence $\Cong $ defined by declaring $a\equiv b$ when
$a \lmodWLnu b $ and $b \lmodWLnu a$, yields an ordered monoid.
\end{rem}

\subsubsection{Surpassing morphisms}

 We also weaken the notion of layered homomorphism for layered \semirings0.

\begin{defn}\label{gsmorph} A \textbf{surpassing map}  $\vrp: R \to R'$ is a (multiplicative) monoid homomorphism  such that $\vrp (a+b) \lmodWLnu \vrp (a) + \vrp (b) .$

A \textbf{surpassed map}  $\vrp: R \to R'$ is a  monoid homomorphism  such that $ \vrp (a) + \vrp (b) \lmodWLnu \vrp (a+b) .$

A \textbf{0-excepted homomorphism}  $\vrp: R \to R'$ is a  monoid
homomorphism such that  $ \vrp (a) + \vrp (b) = \vrp (a+b) $
whenever $s(a), s(b)>0. $
 \end{defn}

(In other words, a 0-excepted homomorphism could fail to be a
 \semiring0 homomorphism only because of the behavior of the 0 sort.)

We write   $R_{\ge \ell}$ for
 $\cup _{k \ge \ell}\,  R_k$.


\begin{example}\label{Fro00}  If $a >_ \nu  b$, then $(a+b)^m = a^m $. Hence, for any given $m$,
the Frobenius property \begin{equation}\label{eq:Frobenius}
(a+b)^m \lmodWLnu a^m + b^m\end{equation} from
\cite[Remark~5.26]{IzhakianKnebuschRowen2009Refined} is satisfied
in any $L$-layered \semiring0 and,  the Frobenius map $a \mapsto
a^m $ is  a surpassing map in $R_{\ge 1} \cup R_0 $.
 \end{example}

\begin{prop} We have the surpassing map $\varphi: M_n(R) \to M_n(R)$ given
by $(a_{i,j}) \mapsto ({a_{i,j}}^m).$
\end{prop}
\begin{proof}  We need to show that $({c_{i,j}}^m) =({a_{i,j}}^m)({b_{i,j}}^m) ,$
where $c_{i,j} = \sum_k  a_{i,k}b_{k,j}.$ But by
\eqref{eq:Frobenius},  $${c_{i,j}}^m = \bigg(\sum_k
a_{i,k}b_{k,j}\bigg)^m
 \lmodL \sum_k  (a_{i,k}b_{k,j})^m = \sum_k  {a_{i,k}}^m{b_{k,j}}^m .$$
\end{proof}

\begin{example}  In the standard supertropical situation, the supertropical determinant (i.e., the permanent)
is a surpassing map, by
\cite{IzhakianRowen2008Matrices}.\end{example}

\begin{prop}\label{strt} Any surpassing map $\varphi$ preserves $\nu$, in the   sense that
 if $a \ge _\nu b,$   then $\varphi(a) \ge _\nu
\varphi(b).$ \end{prop}
\begin{proof} $ $
  $\varphi(a) \nucong \varphi(a+b) \lmodWLnu \varphi(a)+ \varphi(b),$ implying
$\varphi(a) \ge _\nu \varphi(b).$
\end{proof}

Nevertheless, we take the morphisms in this category to be the
 0-excepted homomorphisms.

\subsection{Layered morphisms}

Since morphisms lie at the heart of category theory, the time has
come to consider the morphisms that arise for layered \semirings0.

\begin{defn}
A \textbf{layered morphism} of  $L$-layered \semirings0 is a map
\begin{equation}\label{eq:layMor} \Phi:=(\vrp,\rho): \RLsnu
\to\RLsnuT \end{equation} where $\rho:L \to L' $ is a \semiring0
 homomorphism,
 together with a 0-excepted homomorphism $\vrp: R \to R'$ such that
\boxtext{
\begin{enumerate}

  \item[M1.]  $\lv '(\vrp
 (a)) \ge \rho (\lv(a))$ or   $\lv '(\vrp
 (a)) = 0.$ \pSkip

\item[M2.]  For all $a \in R_k,$   $\lv(\varphi(\nu_{\ell,k})(a)) \nucong  \lv(\varphi(a))$ for all $\ell \ge k.$
\pSkip

  \item[M3.]   If  $a \cong
_ {\nu}b$,
 then  $\vrp(a)  \cong _ {\nu} \vrp(b)$
 (taken in the context of the ${\nu}'_{m',\ell'}$).
\end{enumerate} }
A \textbf{layered homomorphism}  is a layered morphism
  such that $\vrp: R \to R'$ is    a \semiring0 homomorphism.
\end{defn}

We always write $\Phi := (\vrp, \rho ): \RLsnu \to\RLsnuT,$
denoted as $\Phi: R \to R'$ when unambiguous. In most of the
following examples, the sorting \semirings0 $L$ and~$L'$ are the
same.

\begin{example}\label{exmp:layredStr} Here are some examples of   layered
homomorphisms.
 We assume throughout that $R$ is an $L$-layered \semiring0, although sometimes we consider the role of $\rzero$ if it exists. \begin{enumerate}
\ealph  \ddispace \item In the max-plus situation, when $L = \{ 1
\},$ $\rho$ must be the identity, and $\Phi$ is just a \semiring0
homomorphism. When $L = \{ 0, 1 \}$ and $R_0 = \{\rzero\},$ we
must have $\vrp(\rzero) = \rzero.$

\item In the ``standard supertropical situation without 0,'' when
$L = \{ 1, \infty \},$ $\Phi(R_\infty) = R_\infty$.

\item In the ``standard supertropical situation with 0,'' when $L
= \{0, 1, \infty \},$ and $R_0 = \{\rzero \}$, $\Phi$~must send
the ghost layer $R_\infty$ to $R_\infty \cup R_0$. If $\mfa
\triangleleft R$ and $\mfa \supset R_\infty,$ one could take $R' =
R$ as a set, with $R_1' = R_1\setminus \mfa$ and $R'_0 = \mfa$.
The identity map is clearly a layered homomorphism; its
application ``expands the zero level'' to~$\mfa.$

\item Notation as in Theorem~\ref {maxplus23},
we define a layered homomorphism $\scrR(L,
 \tG)\to
\scrR(L,
 \tG)_\mfa$ given by the identity map on all elements of $\scrR(L,
 \tG) \setminus \mfa,$ and $ \xl{a }{\ell  } \mapsto   \xl{a }{0
 }$ for every  $a \in \mfa.$

\item Any \semiring0 homomorphism $\rho: L \to L'$ induces a
layered homomorphism $\scrR(L,
 \tG) \to \scrR(L', \tG) $ given by $\xl{a}{\ell }\mapsto \xl{a}{\rho(\ell) }.$

  \item The natural injections $R_{\ge 1
}\cup R_0 \to R$ and $\{ \bigcup_\ell {R_\ell: \ell \in \Net} \}
\to R$
  are both examples of layered
homomorphisms.

\item The truncation maps of Example~\ref{trun0001} and
Example~\ref{trun0002} are layered homomorphisms.

\item Suppose $\mfa \triangleleft R$ is   a $\nu$-``upper'' ideal
in $R_{\ge 1}$
 or in $R_{\ge 1 }\cup R_0,$ by which we mean an ideal of the form   $\{r:  r \ge_\nu
a\}$ or   $\{r:  r >_\nu a\}$.  We define the congruence $\Cong
_\mfa$  on $R_\mfa$ to be $(\mfa \times \mfa) \cup \diag (R)$; in
other words,  $b_1 \equiv _\mfa b_2$ if $b_1 \nucong b_2$ or if
$b_1,b_2 \in \mfa.$ Then $ {R_\mfa}/\Cong_\mfa$ is a layered
\semiring0, under the induced multiplication and addition of
equivalence classes, and $a \mapsto [a]$ defines a layered
homomorphism. Note that all elements of $\mfa$ collapse to a
single element, as in the Rees quotient construction for
semigroups.

 \end{enumerate}
\end{example}

Having these examples in hand, one might wonder why we bother with
0-excepted homomorphisms in the definition of morphism. This is in
order to make Theorem~\ref{fun1} possible.

\begin{prop}\label{strt} Any layered  morphism $\varphi$ on a tangibly generated layered \semiring0 is determined by its action on the tangible submonoid $ R_0 \cup R_1$. \end{prop}
\begin{proof} $ $ Since  $\varphi (e_k a) = \varphi (e_k)\varphi ( a),$
it suffices to check that $\varphi (e_k) $ is uniquely defined.
Write  $$e_k' = \varphi (\rone) + \dots +\varphi (\rone) , $$
taken $k$ times, whose sort is $k$. Since $\varphi (\rone) = \one_{R'},$ we have
  $e_k' \lmodWLnu   \varphi (e_k)$  by definition of 0-excepted homomorphism.
Hence, $s( \varphi (e_k)) \le k.$ But $s( \varphi (e_k))\ge k$ by Condition M1, implying $s( \varphi (e_k))= k$, and thus
   $e_k' =  \varphi (e_k)$, as desired.
\end{proof}

\section{The layered categories and their tropicalization functors}\label{supfun}

Having assembled the basic concepts, we are finally ready
 to tie these ideas to tropicalization, by introducing the
layered categories. Our objective in this section is to introduce
the functor that passes from the ``classical algebraic world'' of
integral domains with valuation to the ``layered world,''
  taking the cue from
\cite[Definition~2.1]{IzhakianRowen2007SuperTropical}, which we
recall and restate more formally.

 \subsection{Identifications of categories of valued monoids and layered
\semirings0}

Here is our   main layered category.

\begin{defn}\label{somecat} $ $
 $\LaySR$ is   the  category   whose objects are  tangibly generated
layered
 \semirings0
 and whose morphisms are
 layered morphisms.
\end{defn}

\begin{rem}\label{forg2}  In view of Theorem~\ref{maxplus23} we can define the forgetful functor $ \LaySR \to \SOMon$ given by
sending the $L$-layered \semiring0 $R$ to $ R_0 \cup R_1.$
\end{rem}

Thus, any layered homomorphism yields a homomorphism of the
underlying monoid of tangible elements, thereby indicating an
identification between categories arising from the construction of
 layered \predomains0 from  ordered monoids (and more generally,
of layered \semirings0 from valued monoids).
But to get the other direction, we need to permit morphisms merely
to be surpassed maps, as previously defined.

\begin{thm}\label{fun1} For any valued semiring  $L$,  there is a faithful \textbf{layering functor}  $\tF: \Valmon
\to \LaySR$, given by sending $\tM$ to $\scrR(L,\tM)_{\bar
\mfa},$ where $\mfa$ is the monoid ideal of noncancellative
products, and the ordered homomorphism $\vrp: \tM \to \tM'$ to the
layered homomorphism $\tF \vrp: \scrR(L,\tM)\to \scrR(L,\tM')$ obtained from $\vrp$ as follows:

$\tF \vrp  $ is  defined on  $ R_0 \cup R_1$ via $\tF \vrp (
\xl{a}{\ell }) = \xl{\vrp(a) }{\ell' },$ where $\ell' = 1$ unless
$\vrp(a)$ is a noncancellative product in $\tG'$, in which case
$\ell' = 0.$

The functor $\tF$ is a left retract of the forgetful function of
Remark~\ref{forg2}.
\end{thm}
\begin{proof} The image of an ordered monoid $\tG$ is a layered
\semiring0, in view of Proposition~\ref{maxplus2}, and one sees
easily that $\tF \vrp $ is a layered morphism since, for $a
\ge_\nu b,$
$$\tF \vrp ( \xl{a }{k } +  \xl{b}{\ell })
\nucong \tF \vrp ( \xl{a }{k }) \nucong \vrp ( \xl{a }{k }) \nucong \vrp ( \xl{a }{k })+\vrp (
\xl{b }{\ell }),$$ and $s'(\tF \vrp ( \xl{a }{k }
))$ is $ k$ or $0.$

One needs to verify  that  $\bar \mfa R_1
\subseteq \bar \mfa.$ But $\mfa R_1 \subseteq \mfa$ is clear by
definition of noncancellative product, yielding $\bar \mfa R_1
\subseteq \bar \mfa$.

The morphisms match. The functor $\tF$ is faithful, since one
recovers the original objects and morphisms by applying the
forgetful functor of Remark~\ref{forg2}.
\end{proof}


\subsubsection{The layered tropicalization functor}\label{RedTro}

\begin{defn}\label{def:Tfunctor} Given a \semiring0 $L$, the $L$-\textbf{tropicalization functor}
$$\TropfunoneL : \Valmon \To \LaySR$$ from the category of valued
monoids to the category of uniform layered \semirings0 is defined
as follows: $\TropfunoneL: (\tM,\tG,v) \mapsto \scrR(L,\tG)_\mfa $ and
 $\TropfunoneL: \phi \mapsto  \al_\phi ,$
   where $\bfa$ is the ideal of noncancellative elements of the monoid
   $\tG,$ and,
 given a morphism $\phi :  (\tM,\tG,v)  \to  (\tM',\tG',v') $ we define
 $\al_\phi : \scrR(L,\tG) \to \scrR(L',\tG') $, by \begin{equation}\label{morph2}
\al_\phi (\xl{a }{\ell }) := \xl{\phi(a)}{k}, \qquad a \in \tG,
\end{equation}
where $k=0$ if $\phi(a)$ is noncancellative and $k=\ell$ if
$\phi(a)$ is cancellative, cf. Formula \eqref{eq:valMonMor}.
\end{defn}

Note that the $L$-tropicalization functor $\TropfunoneL$  factors
as
$$\Valmon \ds \to \OMon \ds{\to} \zLaySR$$ which restricts to  $\SValmon
\to \SOMon \to \LaySR$ of  \cite{IKR4}.

Suppose $v: W^\times  \to \tG$ is a valuation on an integral
domain $W$, where $W^\times : =  W\setminus \{ \zero_W \}$. Let  $\tM := W^\times $,  a multiplicative monoid. Fix $\ell\in L;$
usually $\ell = 1.$ The restriction of $v$ to $\tM$, which we
denote as $\psil$, can be realized as the map sending $\tM$ as a
set into the $\ell$-layer of $\scrR(L,\tG)$, given by $\psil:
a \mapsto \xl{v(a)}{k}$, where $k= 0$ if $a$ is a noncancellative
product and $k = \ell$ otherwise. This is not  a homomorphism of
\semirings0, since  $a + (-a) = \zero_W$ whereas $v(-a) = v(a)$,
and thus
$$\psil (a + (-a)) =\psil (\zero_W) = \rzero  \neq \xl{a }{2 \ell }  =
\psil (a) + \psil (a) = \psil (a) + \psil (-a)  .$$ But this is
exactly where the layered theory acts more categorically than the
 the max-plus theory.

 \begin{prop}\label{K1} Suppose $W$ is an integral domain
 with valuation  $v$,  and $$\psil : \tM \to \scrR(L,\tG)_\mfa,$$ is
the map just described. If $\sum_i a_i =
 \zero_W$ with each $a_i$ in $\Wstr ,$ then $s \big(\sum_i \psil(a_i) \big) \ge 2$.\end{prop}

 \begin{proof} This is really a reformulation of a standard, elementary fact in   valuation theory, in which we recall that
 $v(\zero_W)$ is undefined.
 It is well-known that if $\sum_i a_i = \zero_W$ then there exist $i_1,
 i_2, \dots$ such that $v(a_{i_1}) = v(a_{i_2})= \dots$ which dominate all other
 $v(a_i)$, since if a single $v(a_{i_1})$ dominated, we would have
 $ \zero_W  = v(\sum_i a_i) =  v(a_{i_1}),$ a contradiction. Hence,
 $$ s\bigg(\sum_i \psil(a_i)\bigg) = s\big (\psil(a_{i_1}))  + s(\psil(a_{i_2})) + \dots \ge 1+1 + \dots \ge 2.$$
\end{proof}

Thus, we see that the $L$-tropicalization functor explains the
importance of the ``surpassing
$L$-relation.''

\subsubsection{The role of Kapranov's Lemma}

We are ready to extend the considerations of \cite[\S8.1]{IKR4}.
%
Since Puiseux series play such an important role in tropical
geometry, let us understand them in terms of layers.

\begin{rem}\label{Puis} We start with a triple $(F,\tG,v)$, where $F$ for example may
be the algebra of Puiseux series, $\tG$ an ordered monoid,  and
$v: F \to \tG$.

Take the layered \semiring0 $R := \scrR(L,\tG).$ Define a
\textbf{Kapranov map} to be a \ZO-supervaluation satisfying the
property:

\begin{equation}\label{Puis2} \Phiv(a)+ \Phiv(b)\lmodL \Phiv(a+b).\end{equation}

This is the analog of the iq-supervaluation in
\cite[Definition~11.12]{IzhakianKnebuschRowen2009Valuation}. By Proposition~\ref{K1}, we see that the Kapranov map sends any
root of $f$ to a corner root of  $\Phiv(f)$.
 This general framework of Kapranov's lemma encompasses
  tropicalizations of finite Puiseux
series introduced in \cite{IzhakianKnebuschRowen2009Refined} and
\cite{IKR4}.
\end{rem}

\section{Layered supervaluations and transmissions:\\ an alternative approach to morphisms}\label{layval}

 In this section we delve deeper into the nature of
morphisms, towards what would be the ``correct'' general
definition in the
 category of layered \semirings0, paralleling the general theory of m-valuations given in \cite{IzhakianKnebuschRowen2009Refined}.
 The outcome is somewhat technical, but enables us to define a functor from the
functions in the algebraic world to the category of layered
function \semirings0, and indicates that Payne's methods
\cite{Pay1} should also be applicable in the layered theory.

In Corollary~\ref{thm5.411}, we will see that this approach
reduces to Section \ref{supfun} in many cases.

 Since
valuations play such an important role, we would like to extend
our definition of morphism to include all maps preserving
valuations. This route leads us to a layered version of
supervaluations and transmissions. See
 \cite{IzhakianKnebuschRowen2009Valuation},
 \cite{IKR3}, \cite{IKR5}
 for further details in the
 supertropical case.


\begin{defn}\label{layeredv} An   $L$-\textbf{layered supervaluation} on a ring $W$,  with respect to a
semiring $L$, is a map $\Phiv: W\to R$
  from $W$ to an $L$-layered semiring~$R$ satisfying the following
  properties.
 \begin{alignat*}{2}
&\LV1:\ &&\Phiv(\wone)=\rone,\\
&\LV2:\ &&\forall a,b\in R: \Phiv(ab)=\Phiv(a)\Phiv(b),\\
&\LV3:\ &&\forall a,b\in R: \Phiv(a+b) \le _\nu  \Phiv(a)+\Phiv(b) ,\\
&\LV4:\ &&\Phiv(\wzero)=\rzero.
\end{alignat*}
A   \ZO-\textbf{supervaluation} on a ring $W$ is an $L$-layered
supervaluation $\tlv: W  \to R$ such that $\tlv(W) \subseteq R_0
\cup R_1.$

An   $L$-\textbf{layered supervaluation$^\dagger$} on an integral
domain $W$,   with respect to a \semiring0 $L$,  is a map $\Phiv:
\Wstr \to R$
  from $\Wstr := W\setminus \{ \wzero  \}$ to an $L$-layered pre-\domain0 $R$ with the following
  properties.
 \begin{alignat*}{2}
&\LV1^\dagger :\ &&\tlv(\wone)=\rone,\\
&\LV2^\dagger :\ &&\forall a,b\in R: \tlv(ab)=\tlv(a)\tlv(b),\\
&\LV3^\dagger :\ &&\forall a,b\in R  : \tlv(a+b)\le _\nu
\tlv(a)+\tlv(b) .\end{alignat*}

\end{defn}

To encompass the results of \cite{IzhakianKnebuschRowen2009Valuation} and \cite{IKR3},   instead of using layered homomorphisms for our morphisms,
we  need to consider a ``transmissive'' property analogous to the
one given in \cite[Definition~4.3]{IKR3}.
%
%
\begin{defn}\label{defn5.100} If $\Phiv: W\to R$ and $\Phiw: W\to
R'$ are $L$-layered supervaluations, where $R$ has sorting map
$s:R \to L$ and $R'$ has sorting map  $s': R' \to L$, we say that
$\Phiv$ \bfem{dominates} $\Phiw$  if the following properties hold for any
$a,b\in W$:
\begin{alignat*}{3}
&\D1.\quad && \Phiv(a)=\Phiv(b) \quad\Rightarrow\quad \Phiw(a)=\Phiw(b),\\
&\D2.\quad &&\Phiv(a)\le_\nu \Phiv(b)\quad \Rightarrow\quad \Phiw(a)\le_\nu  \Phiw(b),\\
&\D3.\quad &&\Phiv(a)\in R_0 \quad \Rightarrow\quad  \Phiw(a)\in R'_0,\\
&\D4.\quad && s(\Phiv(a)) \le s'( \Phiw(a)) \quad \text{whenever}
\quad \Phiw(a)\notin R'_0.
\end{alignat*}
(We   omit $\D3$ and the condition in  $\D4$ for layered
supervaluations$^\dagger$, since we do not need to bother with the
0 layer.)
\end{defn}
\begin{defn} For $L$-layered \domains0 $R$ and $R'$ and $\tM \subset R,$ a map $\al : \tM\to R'$  is
$\nu$-\textbf{preserving}  if $$a \le _\nu b \quad \text{implies}
 \quad \a (a) \le_{\nu} \a (b)$$ for all $a,b \in R.$
\end{defn}

\begin{lem} For  any  $\nu$-preserving map $\a$, if $a \nucong b$, then $\a (a)
\cong_{\nu} \a (b)$ for  $a,b \in R.$
\end{lem}
\begin{proof} $a \nucong b$ implies $\al   (a)\le _\nu
\al  (b)$ and likewise
 $\al   (b)\le _\nu
\al  (a)$, so $\al  (a)\nucong  \al  (b).$ \end{proof}

\begin{lem}\label{lem5.2} Let $\Phiv: W\to R$ and $\Phiw: W\to R'$
be $L$-layered supervaluations.
If
  $\Phiv$~dominates $\Phiw$, then there exists a unique $\nu$-preserving map
$\al _{\Phiw,\Phiv}:  \Phiv(W)\to R'$ with
$\Phiw= \al _{\Phiw,\Phiv}\circ \Phiv$.
\end{lem}

\begin{proof} By $\D1$ we have a well-defined map
$\al _{\Phiw,\Phiv}: \Phiv(W)\to\Phiw(W)$ given by
$\al _{\Phiw,\Phiv}(\Phiv(a))=\Phiw(a)$ for all $a\in W$.
Furthermore, if $\Phiv(a)\le_{\nu} \Phiv(b)$, then $\D2$ implies
$\Phiw(a)\le _\nu \Phiw(b)$,  so $\al _{\Phiw,\Phiv}$ is
$\nu$-preserving.\end{proof}

\begin{defn}\label{defn5.5a} For layered semirings $R$ and ${R'}$,  a
\textbf{transmission} from $R$ to ${R'}$  is a $\nu$-preserving
map $\al : \tM \to {R'}$, with  $\tM$ a multiplicative submonoid
of $R$, 
satisfying
the following axioms:
\begin{alignat*}{2}
&\TMM 1: \quad \al (\rone)= \one_{R'}, &&\\
 &\TMM 2: \quad
\al (ab)=\al (a)\al (b), \quad &&\forall a,b\in R, \\
&\TMM 3: \quad \al (a+b) \nucong \al (a)+\al (b), \quad &&
\text{whenever}\quad a ,\ b
 ,\ a\!+ \!b\ \in \tM.
 \end{alignat*}
Axioms $\TMM 1$ and $\TMM 2$ imply that $\a$ is a monoid homomorphism, which
we denote as $\a: (R,\tM)\to R'$ to emphasize that $\tM$ is a
submonoid of $R$.
 We write $\tM_\ell$ for $R_\ell \cap \tM.$ A \ZO-\textbf{transmission} from $R$ to ${R'}$ is a {transmission}
 $\a : (R,\tM)\to R'$ for which $\a (\tM_1) \subseteq R'_1 \cup R'_0.$ 
\end{defn}

\begin{lem}\label{patch}
Axiom $\TMM 3$ is equivalent to the map $\a$ being
$\nu$-preserving.\end{lem}
 \begin{proof} $(\Rightarrow)$  If $a \le _\nu b,$ then
 $$\a (b)   \nucong \a (a+b)   \nucong \a (a) + \a(b),$$
 implying $\a (a) \le_\nu \a(b).$

 $(\Leftarrow)$ We may assume that $a \le _\nu b,$ implying $a+b \nucong b.$ Then $\a (a) \le_\nu \a (b),$ so $$  \a (a) +\a (b)
  \nucong \a (b)   \nucong \a (a+b).$$
\end{proof}

Note that the condition of the lemma does not refer explicitly to
calculating sums in $\tM$, so we can study transmissions without
worrying about addition on $\tM.$

 \begin{thm}\label{thm5.4}
 Let $\Phiv: W\to R$ be an $L$-layered supervaluation and $\Phiw:
 W\to {R'}$ an $L$-layered supervaluation dominated by $\Phiv.$ The map
 $\al :=\al _{\Phiw,\Phiv}: (R,\Phiv(W))\to R'$ is a transmission from $R$ to ${R'}$.

Conversely,  assume that $\Phiv: W\to R$ is an $L$-layered
supervaluation and $\al : \Phiv(W)\to R'$ is a transmission from
$R$ to an $L$-layered \semiring0 $R'$. Then $\al \circ \Phiv:
W\to R'$ is   an $L$-layered supervaluation dominated by $\Phiv.$
 \end{thm}

 \begin{proof}

 TM1  and $\TMM 2$  are obtained from the construction of $\al _{\Phiw,\Phiv}$
in the proof of  Lemma~\ref{lem5.2}. 
%
Now assume that $a \le_\nu b$, so $\Phiv(a)\le_\nu\Phiv(b),$ and
thus $\Phiv(a)+ \Phiv(b)  \nucong \Phiv(b).$  But
$\Phiw(a)\le_\nu\Phiw(b)$ by $\D2$, so
$$ \al (\Phiv(a))+\al (\Phiv(b))  = \Phiw(a)+ \Phiw(b) \nucong   \Phiw(b) = \al (\Phiv(b))
\nucong  \al ( \Phiv(a)+\Phiv(b))  .$$
 This is $ \TMM 3$.

For the reverse direction, let $\Phiw: =\al \circ\Phiv,$ Clearly
$\Phiw$ inherits the properties $\LV1$--$\LV3$ from $\Phiv,$ since
$\al $ satisfies $\TMM1$--$\TMM3$.
\end{proof}

\begin{cor}\label{transpres} Every transmission of Theorem~\ref{thm5.4} is $\nu$-preserving.
\end{cor}
\begin{proof} $\a$ is the map of Lemma~\ref{lem5.2}, so is $\nu$-preserving.
\end{proof}


 It is evident
that every \semiring0 homomorphism from   $R$ to ${R'}$ is a
transmission, but there exist  transmissions that are not
\semiring0 homomorphisms; cf.~\cite[\S9]{IzhakianKnebuschRowen2009Valuation}. Nevertheless, we do get
\semiring0 homomorphisms in the following basic case. We say that
the transmission $\a$ is \textbf{homomorphic} if it satisfies the
condition
\begin{equation}\label{add1} \al (a+b) = \al (a) +
\al (b)\end{equation} whenever $a,b, a+b \in \tM.$


  Every homomorphic transmission satisfying $\tM = R$ is a layered
 homomorphism, by definition.  We say that a
$\nu$-preserving map
 $\a$ is  \textbf{strictly}
$\nu$-\textbf{preserving} if $a < _\nu b$ implies that either
$\a(a) \in R'_0$ or $\a (a) <_{\nu'} \a (b)$.

 \begin{thm}\label{thm5.41}
 Let $\Phiv: W\to R$ be an \ZO-layered supervaluation and $\Phiw:
 W\to {R'}$ an \ZO-layered supervaluation dominated by $\Phiv.$
 Then the
 \ZO-transmission $\al :=\al _{\Phiw,\Phiv}:(R,  \Phiv(W))\to R'$ is  homomorphic,
 iff it is strictly $\nu$-preserving.
\end{thm}
 \begin{proof}
$(\Rightarrow)$ Follows from Corollary~\ref{transpres}.

$(\Leftarrow)$  We need to check \eqref{add1}.
 If  $a <_\nu b$, then $\al (a+b) = \al (b),$ so
\eqref{add1}
  holds iff $\al (a) <_\nu \al (b)$ or $\a(a) \in
  R_0$. The symmetric argument holds when $b <_\nu a$. Finally, if $a  \nucong
  b$, with $a \in R_0,$ then $\a (a) \in R_0,$ with $\a (a) \nucong \a (b),$ so
  $$\a (a+b) = \a (b) = \a (a) + \a(b).$$
  Likewise for $b \in R_0,$ so we may assume that $a,b \in R_1.$
  Then $a+b \in R_2,$ so there is nothing to check.
 \end{proof}

 \begin{cor}\label{thm5.411}
 Suppose $\Phiv: W\to R$ is an \ZO-layered supervaluation such that $\Phiv(W)$ strictly
 generates~$R$, and $\Phiw:
 W\to {R'}$ is an \ZO-layered supervaluation dominated by $\Phiv.$  Then the
 \ZO-transmission  $\al :=\al _{\Phiw,\Phiv}:(  R, \Phiv(W))\to R'$ extends to a layered homomorphism from $R$ to $R'$,
 iff  $\al $ is  strictly
$\nu$-preserving.

 In particular,  when $R$ is  uniform, every $\{0,1\}$-transmission yields a layered homomorphism.
\end{cor}

 \begin{rem}\label{def:cqtgen}
Since every transmission is a monoid homomorphism, we have a
subcategory L-STROP of the category of monoids and monoid
homomorphisms, whose objects are  layered semirings $\RLsnu$,  and
whose morphisms are the $\{0,1\}$-transmissions. Explicitly,
 $\{0,1\}$-transmissions    from $R$ to ${R'}$  and from $R'$ to ${R''}$
 are described respectively  as
{transmissions}
 $\a : (R,\tM)\to R'$ for which $\a (\tM_1) \subseteq R'_1 \cup R'_0$
 and
 $\a ': (R',\tM')\to R'$ for which $\a' (\tM_1') \subseteq R''_1 \cup R''_0.$
 Their composition can be defined as $\a' \circ \a$ whenever $\a (\tM) \subseteq
 \tM_1'$.

This category closely resembles the category STROP of \cite{IKR3}
(but with the subtle difference indicated in~Remark~\ref{inf0}),
and encompasses the category from~\S\ref{somecat}.
\end{rem}

\section{Appendix: Layered monoids}\label{mor2}

At times we do not want  additivity at the 0 level, since the
vagaries of cancellation complicate the statements and proofs some
of the theorems. But then we must give up addition between $R_0$
and other levels. At this generality, our next structure is not
quite a \semiring0, since distributivity does not hold at the
0-layer, but we copy what we can from Definition~\ref{defn10}.

\begin{defn}\label{defn1} Suppose  $(\Lz, \ge)$ is a directed, partially pre-ordered
\semiring0. An $\Lz$-\textbf{layered monoid} $$\R :=\ldsR, \qquad
$$ is a multiplicative monoid  $\R $ which is a disjoint union of subsets $R_\ell,$ $\ell \in L$, together
with addition defined on  $ R_{0}$ and on $ R_{>0} := \dot \bigcup_{\ell >0}R
_\ell$ such that
\begin{equation}\label{unionp} R = \dot \bigcup_{\ell\in \Lz} R
_\ell,\end{equation}  together with a family
  of \textbf{sort transition maps}
$$   \nu_{m,\ell}:R _\ell\to R _m,\quad
\forall m\ge \ell >0 ,$$  such that $$\nu_{\ell,\ell}=\id_{R
_\ell}$$ for every $\ell\in \Lz,$ and
$$\nu_{m,\ell}\circ \nu_{\ell,k}=\nu_{m,k} , \qquad  \forall m \geq \ell \geq k, $$ whenever both sides
are defined. We  also require the  axioms A1--A4,  and B,  given
presently, to be satisfied.

We  define $R_\infty$ to be the direct limit of the $R _\ell,$ $
\ell
>0 $, together with a map $\nu:R _\ell \to R _\infty$, which
extends to a map   $\nu:R   \to R _\infty$.  We write $a ^\nu$ for
$\nu(a)$.

 We write $a \nucong b$ for $b \in R_\ell$,  whenever  $\nu (a) = \nu ( b ).$ (For $k, \ell >0$ this means
 $\nu_{m,k}(a) = \nu_{m,\ell}( b )$ in $R_m$ for some $m \ge
 k,\ell$. The notation is used generically, as before.)
 Similarly, we write $a \le _\nu b$ if  $a \nucong b$  or $\nu_{m,k}(a) + \nu_{m,\ell}( b )=  \nu_{m,\ell}( b )$ in $R_m$ for some $m \ge
 k,\ell$.

 The axioms are as follows:

\boxtext{
\begin{enumerate}

\item[A1.] $\rone \in R _{1}.$ \pSkip

 \item[A2.] If $a\in R _k$ and  $b\in
R _\ell,$ then $ab\in m$ where $ m\ge k \ell  $ or $m =0$. \pSkip

 \item[A2$'$.] If $a\in R _0$ or  $b\in
R _0,$ then $ab\in R _0$. \pSkip

\item[A3.] The product in $R $ is compatible with sort transition
maps: Suppose $a\in R _\ell,$ $b\in R _{\ell}',$ with $m\ge \ell$
and $m'\ge \ell '.$

 If  $ab\in R_{ \ell ''}$ for $\ell '' \ge \ell ' \ell,$  then
$\nu_{m,\ell}(a)\cdot\nu_{m',\ell '}(b)= \nu_{m'', \ell ''}(ab)$
for some $m'' \ge mm'.$ \pSkip

\item[A4.] $\nu_{\ell,k}(a) + \nu_{\ell',k}(a)
 =\nu_{\ell+\ell',k}(a)  $ for all $a \in R_k$ and all $\ell, \ell' \ge k.$
 \pSkip

 \item[A5.] If $a \in R_k$, $b \in R_{\ell}$,
 and $c = a+b \in R_{k'}$, then $$\nu_{m,k'}(c) = \nu_{m,k}(a) +
 \nu_{m,\ell}(b)$$ for each $m \ge k+\ell$. \pSkip

  \item[A6.] $R_{>0}$ is an additive semigroup  and $R_0$ is an $R_{>0}$-module, in the sense that  $R_0$ is an additive semigroup together with a multiplication  $R_0 \times R_{>0} \to R_0 $ satisfying distributivity and associativity whenever defined.\end{enumerate}}

 \boxtext{
\begin{enumerate}

\item[B.] (Supertropicality) Suppose  $a\in \R _k,$ $b\in \R
_{\ell},$ and $a \nucong b$. Then \\ $a+b \in R_{k+\ell}$ with
$a+b \nucong a$.\end{enumerate}}
%

The \textbf{sorting map} $s:\R\to L,$ is a map that sends every
element $a\in\R_\ell$ to its sort   $\ell$.
%
 \end{defn}

%
%
%

\begin{example}\label{join1}
Suppose $R$ is a  layered pre-\domain0, and define formally $R_0$
to be another copy of   $R $ with~$0$ adjoined in the natural way,
where we write $e_0$ for its multiplicative unit. Then
 $R \cup R_0$ is naturally a  layered monoid, where we
 define $(e_0 a)b := e_0(ab)$ and $e_0 a + e_0 b = e_0
 (a+b).$
\end{example}

\begin{rem} What we are lacking for obtaining a semiring is the definition
of $a+b$ for $a\nucong b$ with $a \in R_0$  and $b\in R_\ell$ for
$\ell>0.$ The natural guess might be to define $a+b=b$ in this
case, but this  could ruin distributivity. If there happens
to be $c \in R$ such that $bc \in R_0,$ then  we would have
$(a+b)c = bc$, which does not necessarily equal $ac + bc.$
\end{rem}

In this generality, we also need a more  intricate definition of
morphism.

\begin{defn}
A \textbf{layered morphism} of tangibly generated $L$-layered
monoids is a map
\begin{equation}\label{eq:layMor} \Phi:=(\vrp,\rho): \RLsnu
\to\RLsnuT \end{equation} such that $\rho:L \to L' $ is a
\semiring0
 homomorphism,
 together with a  multiplicative monoid homomorphism $\vrp: R \to R'$   that also preserves addition on $  R_{>0}$ in the sense that
 $ \vrp (a+b) =
 \vrp (a)+\vrp (b)$ for all $a,b$ in~ $  R_{>0}$, and which also satisfies the following properties:
\boxtext{
\begin{enumerate}

  \item[M1.] If $\vrp(a) \notin R_0',$ then  $\lv '(\vrp
 (a)) \ge \rho (\lv(a))$ or   $\lv '(\vrp
 (a)) = 0.$ \pSkip

  \item[M2.]   $ \vrp(a^\nu)  \nucong  \vrp(a)$. \pSkip

  \item[M3.] If  $a \cong_ {\nu}b$,
 then  $\vrp(a)  \cong _ {\nu} \vrp(b). $

\end{enumerate} }

\end{defn}

The ensuing category closely resembles the category STROP$_m$ of \cite{IKR5}.

\subsection{Weakening the structure of $L$ and $R$}\label{basicsym}


\begin{Note}\label{arith} To generalize the notion
``supertropical semiring'' from the standard supertropical theory,
we could weaken Axiom A2 to: \boxtext{
\begin{enumerate}\item[wA2.] If $a\in
\R _k$ and $b\in \R _\ell,$ then $ab\in \R _m$ for some $m \ge {k
\ell }.$\end{enumerate}}

Now we have to modify Axiom A3 to make it compatible; i.e.,
multiplication commutes with the sort transition maps.
Technically, this  says:

 \boxtext{
\begin{enumerate}\item[wA3.] If $a\in \R _k$ and $a'\in \R _{k'},$ with $aa' \in R_{k''}$ and
$\nu_{\ell,k}(a)\cdot\nu_{\ell',k'}(a')\in R_{\ell''}$ and
$\nu_{m,\ell}(a)\cdot\nu_{m',\ell'}(a'')\in R_{m''},$ for $m \ge
\ell$, $m' \ge \ell'$, and $m'' \ge mm'$,
 then

$\quad
\nu_{q,\ell''}(aa')=\nu_{q,m''}(\nu_{m,\ell}(a)\cdot\nu_{m',\ell'}(a'))$
 for all $q \ge \ell'',m''$.  \end{enumerate}}
\end{Note}

This weakening is of arithmetic interest, since we now have a
version of the theory without requiring a zero layer.

\begin{rem}\label{remove1} We do not need $L$ to be a
\semiring0, but merely a directed, partially pre-ordered
multiplicative monoid (without addition). This material yields an intriguing parallel between the
layered monoid $R$ and the sorting set $L$ (since any ordered monoid
becomes a \semiring0 when addition is taken to be the maximum),
and may provide guidance for future research.

 Since $L$ now is
only assumed to be a multiplicative monoid, we need to remove
references to addition in~ $L$. Thus, we need
 a formal ``doubling function'' $\ell \mapsto 2\ell$ on
$L$,   eliminate Axiom A4, and weaken Axiom~B
to:

\boxtext{
\begin{enumerate}
\item[wB.]  (weak supertropicality) If $a\in \R _k$ and $b\in \R
_{\ell }$ with $a \nucong b$, then $a+b \in R_m$ for some
 $m \ge k, \ell, \min\{2k, 2\ell\}$  with $a+b \nucong b$.
\end{enumerate}}
\end{rem}

%



%

\end{document}